\documentclass[11pt,twoside]{article}
\usepackage{artmods,macros,amsmath,amssymb}
\usepackage{cite}
\usepackage[utf8]{inputenc}

\input macros

\usepackage{float}
\usepackage{graphicx, multirow, xcolor}
\usepackage{algorithm, algorithmicx}
\usepackage[hidelinks]{hyperref}

\begin{document}

\title{A structured modified Newton approach for solving systems of nonlinear equations arising in interior-point methods for quadratic programming}

\author{David Ek\thanks{\footTO} \and Anders Forsgren\addtocounter{footnote}{-1}\footnotemark}


\def\footTO{Optimization and Systems Theory, Department of
  Mathematics, KTH Royal Institute of Technology, SE-100 44 Stockholm,
  Sweden ({\tt daviek@kth.se,andersf@kth.se}).}

\maketitle

\begin{abstract}
 The focus in this work is on interior-point methods for inequality-constrained quadratic programs, and particularly on the system of nonlinear equations to be solved for each value of the barrier parameter. Newton iterations give high quality solutions, but we are interested in modified Newton systems that are computationally less expensive at the expense of lower quality solutions.  We propose a structured modified Newton approach where each modified Jacobian is composed of a previous Jacobian, plus one low-rank update matrix per succeeding iteration. Each update matrix is, for a given rank, chosen such that the distance to the Jacobian at the current iterate is minimized, in both 2-norm and Frobenius norm. The approach is structured in the sense that it preserves the nonzero pattern of the Jacobian. The choice of update matrix is supported by results in an ideal theoretical setting. We also produce numerical results with a basic interior-point implementation to investigate the practical performance within and beyond the theoretical framework. In order to improve performance beyond the theoretical framework, we also motivate and construct two heuristics to be added to the method.
  
\end{abstract}


\section{Introduction}
This work is intended for inequality-constrained quadratic programs on the form
\begin{equation} \tag{IQP} \label{eq:IQP} 
\begin{array}{cl} 
\minimize{} & { \frac{1}{2}  x^T H x + c^T x } \\
 \subject & Ax \geq b,
\end{array} 
\end{equation}
where $H \in \mathbb{R}^{n \times n}$, $A \in\mathbb{R}^{m \times n}$,
$x \in\mathbb{R}^{n}$, $c \in\mathbb{R}^{n}$ and $b
\in\mathbb{R}^{m}$. We consider primal-dual interior-point
methods which means solving or approximately solving a sequence of
systems of nonlinear equations. Application of Newton iterations on each system of nonlinear equations gives an
\emph{unreduced} unsymmetric block 3-by-3 system of linear equations,
with dimension $n+2m$, to be solved at each iteration. This system can
be put on an equivalent form which contains a \emph{reduced} symmetric
block 2-by-2 system of dimension $n+m$, or a \emph{condensed} system of
dimension $n$, see Section~\ref{sec:Background}. Both of these systems
typically become increasingly ill-conditioned as the iterates converge,
whereas the unreduced system may stay well conditioned throughout, see e.g.,~\cite{GreMouOrb2014,MorSimTan2016} for analysis of the spectral
properties of systems arising in interior-point methods for convex quadratic
programs on standard form. The sparsity
structure of the unreduced system is maintained in the reduced system.
However, the condensed system is typically dense if $A$ contains dense
rows. The most computationally expensive part of an interior-point
iteration is the solution of the Newton systems that arise, see
e.g.,~\cite{FGG07,ForGilWri2002} for details on the solution of the
systems and \cite{MorSimTan2017} for a comparison of the solution of unreduced and reduced systems.
  
We are particularly interested in the convergence towards a solution when the Jacobian of each Newton system is modified so that the system becomes computationally less
expensive to solve. In general, there is a trade-off between solving many modified Newton systems, which are computationally less expensive, but typically give lower quality solutions and solving Newton systems which give high quality solutions. We propose an approach where each modified Jacobian is composed of a Jacobian at a previous iteration, whose factorization is assumed to be known, plus one low-rank update matrix per succeeding iteration. A similar strategy have been studied by Gondzio and Sobral~\cite{GonSob2019} in the context of quasi-Newton approaches, where Broyden's rank-1 updates are performed on the Jacobian approximation. If the proposed quasi-Newton approach is started with an exact Jacobian, then the sparsity pattern of the first two block rows is maintained, however the sparsity pattern of third block row is typically lost. In contrast, we consider low-rank update matrices of variable rank which capture the sparsity pattern of all block rows of the Jacobian. Each modified Jacobian may hence be viewed as a Jacobian at a different point. Consequently, the modified Newton approach may also be interpreted in the framework of previous work
on e.g.,~effects of finite-precision arithmetic, stability, convergence and solution techniques for interior-point methods
\cite{Wri2001,Wri1999,Wri1997,Wri1995s,ForGilShi1996,MorSimTan2017,GreMouOrb2014,ApuSimSer2010,MorSimTan2016,FGG07,MorSim2017}. The
idea of low-rank update matrices in the context of a primal barrier
method for linear programming has been considered by Gonzaga~\cite{Gonzaga1989}.

The updates and the theory are given for the unreduced Jacobian but we also discuss how analogous updates can be made on both reduced and condensed systems. The modified Newton approach is also compatible with certain regularization strategies, see e.g.,~\cite{FriOrb2012,AltGon1999,SauTom1996}, although it is outside the scope of this first study.  

The work is meant to be an initial study of the structured modified Newton approach. We derive theoretical results in ideal settings to support the choice of update matrix. In addition, we produce numerical results with a basic interior-point algorithm to investigate the practical performance within and beyond the theoretical framework. The numerical simulations were performed on benchmark problems from the repository of convex quadratic programming problems by Maros and M\'{e}sz\'{a}ros \cite{MarMes1999}. We envisage the use of the modified Newton approach as an accelerator to a Newton approach.
E.g.,~when these can be run in parallel for a specific value of the barrier parameter, and the modified Newton approach may utilize factorizations from the Newton approach when it is appropriate.

The manuscript is organized as follows; Section~\ref{sec:Background}
contains a brief background to primal-dual interior-point methods and
an introduction to the theoretical framework; in
Section~\ref{sec:lowRankUpdMat} we propose a modified Newton approach
and discuss how it relates to some previous work on interior-point
methods; Section~\ref{sec:impl} contains a description of the
implementation along with two heuristics and a refactorization
strategy; in Section~\ref{sec:numRes}, we give numerical results on convex quadratic programs; finally in Section~\ref{sec:conc} we give some concluding remarks.

\subsection{Notation}
Throughout, $\rho(M)$ denotes the spectral radius of a matrix $M$ and
$\vert\mathcal{S}\vert$ denotes the cardinality of a set
$\mathcal{S}$. The notion ``$\cdot$" is defined as the component-wise
multiplication operator, ``$\succ 0$'' as positive definite, and ``$\land$'' as the
logical and. Quantities associated with Newton iterations will
throughout be labeled with `` $\hat{ }$ ''. Vector subscript and
superscript denote component index and iteration index
respectively. The only exception is $e_i$ which denotes the $i$th unit
vector of appropriate dimension. All norms considered are of type
2-norm unless otherwise stated.

\section{Background} \label{sec:Background}

The theoretical setting is analogous to the setting in a previous work of
ours on bound-constrained nonlinear problems \cite{EkFor2021}. For
completeness, we review the background adapted to inequality-constrained quadratic programs in this section. Our interest is
focused on the situation as primal-dual interior-point methods
converge to a local minimizer $x^* \in \mathbb{R}^n$ with its
corresponding multipliers $\lambda^* \in \mathbb{R}^m$ and slack
variables $s^* \in \mathbb{R}^m $. Specifically we assume that the
iterates of the method converge to a vector $\left( x^{*T},
  \lambda^{*T}, s^{*T} \right)^T \triangleq (x^*,\lambda^*,s^*)$ that
satisfies
\begin{subequations} \label{eq:OptCond} 
\begin{align}
Hx^* + c - A^T \lambda^*  = 0, &\quad \textrm{ (stationarity) } \label{eq:OptCond:stationarity} \\
Ax^*-s^*-b   = 0, & \quad \textrm{ (feasibility)} \label{eq:OptCond:feasibility} \\
 s^* \geq 0, & \quad \textrm{ (non-negativity of slack variables)} \label{eq:OptCond:nonnegofslacks}  \\
 \lambda^*  \geq 0, & \quad \textrm{ (non-negativity of multipliers)} \label{eq:OptCond:nonnegofmultipliers}  \\
s^* \cdot \lambda^* = 0, & \quad \textrm{ (complementarity)} \label{eq:OptCond:complementarity} \\
 Z(x^*)^T H Z(x^*) \succ 0, & \label{eq:OptCond:posdef} \\
 s^* + \lambda^* > 0,  & \quad \textrm{ (strict
                   complementarity)} \label{eq:OptCond:strictcomplementarity},
  \\
  A_\mathcal{A}(x^*) \mbox{ of full column rank}, & \quad \textrm{ (regularity)}\label{eq:OptCond:regularity},
\end{align} 
\end{subequations}%
where $A_\mathcal{A}(x^\star)$ denotes the Jacobian corresponding to
the active constraints and $Z(x^*)$ denotes a matrix whose columns
span the nullspace of $A_\mathcal{A}(x^*)$. First-order necessary
conditions for a local minimizer of (\ref{eq:IQP}) are given by
(\ref{eq:OptCond:stationarity})-(\ref{eq:OptCond:complementarity}). The
first-order conditions together with (\ref{eq:OptCond:posdef}) and
(\ref{eq:OptCond:strictcomplementarity}) constitute second-order
sufficient conditions for a local minimizer of (\ref{eq:IQP})
\cite{GriNasSof2009}. In the theoretical results, we also assume that
$(x^*, \lambda^*, s^*)$ satisfies (\ref{eq:OptCond:regularity}).

To simplify the notation, we let $z$ denote the triplet $(x, \lambda,
s)$. For a given barrier parameter $\mu \in \mathbb{R}$, we are
interested in the function $F_{\mu}:\mathbb{R}^{n+2m} \rightarrow
\mathbb{R}^{n+2m}$ given by
\begin{equation}
F_{\mu}(z) = \begin{pmatrix}
Hx +c - A^T\lambda \\
Ax-s-b \\
\Lambdait S e - \mu e
\end{pmatrix}, \mbox{ with }z = (x,\lambda,s),
\end{equation}
where $S=\textrm{diag}(s)$, $\Lambdait = \textrm{diag}(\lambda)$ and $e$ is a vector of ones of appropriate size. First-order necessary conditions for a local minimizer of (\ref{eq:IQP}), (\ref{eq:OptCond:stationarity})-(\ref{eq:OptCond:complementarity}), are satisfied by a vector $z$, with $s\ge0$ and $\lambda\ge0$, that fulfills $F_{\mu}(z) = 0$ for $\mu=0$. In interior-point methods $F_{\mu}(z) = 0$ is solved or approximately solved for a sequence of $\mu>0$, that approaches zero, while preserving $s>0$ and $\lambda > 0$. Application of Newton iterations means solving a sequence of systems of linear equations on the form 
 \begin{equation} \label{eq:PDnewtSyst}
F'(z) 
\Delta \hat{z}
= -F_{\mu}(z),  
\end{equation}
where $\Delta \hat{z} = (\Delta \hat{x}, \Delta \hat{\lambda}, \Delta \hat{s})$ and $F': \mathbb{R}^{n+2m} \rightarrow \mathbb{R}^{(n+2m) \times (n+2m)}$ is the Jacobian of $F_{\mu}(z)$, defined by
\begin{equation}\label{eq:Fp}
F'(z)  = \begin{pmatrix}
H & -A^T & \\
A &      & -I \\
  &    S & \Lambdait
\end{pmatrix}.
\end{equation}
In consequence it follows that the Jacobian at $z+\Delta z$ can be written as
\[
F'(z+\Delta z) = F'(z)+ \Delta F'( \Delta z), 
\]
where 
\begin{equation} \label{eq:dFp}
\Delta F'( \Delta z) = \begin{pmatrix} 0 & 0 & 0 \\ 0 & 0 & 0\\ 0 & \Delta S & \Delta \Lambdait \end{pmatrix},
\end{equation}
with $\Delta \Lambdait = \mbox{diag}(\Delta \lambda)$  and $ \Delta S = \mbox{diag}(\Delta s)$. The subscript $\mu$ has been omitted since $F'$ is independent of the
barrier parameter.  Under the assumption that $\Lambdait$ is nonsingular the \emph{unreduced} block 3-by-3 system (\ref{eq:PDnewtSyst}) can be reformulated as the \emph{reduced} system
\begin{equation} \label{eq:PDnewtSystSym}
\begin{pmatrix}
H & A^T \\ A & -\Lambdait \inv S \end{pmatrix} \begin{pmatrix}  \Delta \hat{x} \\ -\Delta \hat{\lambda} \end{pmatrix}  =- \begin{pmatrix} Hx +c - A^T \lambda \\ Ax -b - \mu \Lambdait\inv e \end{pmatrix},
\end{equation} 
together with $\Delta \hat{s} = -(s - \mu \Lambdait^{-1} e ) -
\Lambdait^{-1} S \Delta \hat{\lambda}$.  If in addition $S$ is
  nonsingular, then a Schur complement reduction of $\Lambdait \inv S$
  in (\ref{eq:PDnewtSystSym}) gives the \emph{condensed} system
\begin{equation}  \label{eq:PDnewtSystCond}
(H + A^T S \inv \Lambdait A ) \Delta \hat{x} = - (Hx + c - A^T \lambda ) - A^T S \inv ( \Lambdait (Ax -b) - \mu e),
\end{equation}
together with $\Delta \hat{\lambda} = - S\inv ( \Lambdait (Ax-b) - \mu e ) - S \inv \Lambdait A \Delta \hat{x}$. The focus in the manuscript will mainly be on the unreduced block 3-by-3 system (\ref{eq:PDnewtSyst}). However, analogous reductions of the modified Newton system similar to those of (\ref{eq:PDnewtSystSym}) and (\ref{eq:PDnewtSystCond}) will also be discussed.

To improve efficiency, many methods seek approximate solutions of $F_\mu (z)= 0$, for each $\mu$. There are different strategies to update $\mu$, e.g., dynamically every iteration or to keep $\mu$ fixed until sufficient decrease of a merit function is achieved. Herein, our model method uses the latter. In particular, our model method is similar to the basic interior-point method of Algorithm~19.1 in Nocedal and Wright \cite[Ch.~19, p.~567]{NocWri06}. However, termination and reduction of $\mu$ are based on the merit function $\phi_\mu(x) = \| F_{\mu}(z) \|$. 
Similarly, in the theoretical framework we consider the basic condition $\| F_{\mu} (z) \| \leq C\mu$, for some constant $C>0$, see e.g.,~\cite[Ch.~17, p.~572]{NocWri06}. The
additional assumption that all vectors $z$ satisfy $s>0$ and $\lambda > 0$ is made throughout.

In the remaining part of this section we give some definitions and provide the details for the theoretical framework.  
\begin{definition}[Order-notation]
Let $\alpha$, $\gamma \in \mathbb{R}$ be two positive related quantities. If there exists a constant $C_1>0$ such that $\gamma \geq C_1 \alpha$ for sufficiently small $\alpha$, then $\gamma = \Omega(\alpha)$. Similarly, if there exists a constant $C_2>0$ such that $\gamma \leq C_2 \alpha$ for sufficiently small $\alpha$, then $\gamma = \mathcal{O}(\alpha)$. If there exist constants $C_1, C_2  > 0$ such that $C_1 \alpha \leq \gamma \leq C_2  \alpha$ for sufficiently small $\alpha$, then $\gamma = \Theta(\alpha)$.
\end{definition}
\begin{definition}[Neighborhood] For a given $\delta >0$, let the neighborhood around $z^*$ be defined by $\mathcal{B}( z^*, \delta) = \{ z : \| z-z^* \|  < \delta \}$. 
\end{definition}
\begin{assumption}[Strict local minimizer] The vector $z^*$ satisfies
  (\ref{eq:OptCond}), i.e., second-order sufficient optimality
  conditions, strict complementarity and regularity hold. \label{ass1}
 \end{assumption}
The first of the following two lemmas provides the existence of a neighborhood where the Jacobian is nonsingular. The second lemma gives the existence of a Lipschitz continuous barrier trajectory $z^\mu$ in the neighborhood where the Jacobian is nonsingular. The results are well known and can be found in e.g., Ortega and Rheinboldt \cite{OrtRhe00}. See also Byrd, Liu and Nocedal \cite{ByrLuiNoc98} for the corresponding results in a setting similar to the one considered here.
\begin{lemma} \label{lemma:background:FpnonsingCont}
Under Assumption~\ref{ass1} there exists $\delta>0$ such that  $F'(z)$ is continuous and nonsingular for $z \in \mathcal{B}(z^*, \delta)$ and 
\begin{equation*}
\| F'(z) \inv \|  \leq M,
\end{equation*}
for some constant $M>0$.
\end{lemma}
\begin{proof}
See \cite[p.~46]{OrtRhe00}. 
\end{proof}

 \begin{lemma} \label{lemma:LipcPath}
Let Assumption~\ref{ass1} hold and let $\mathcal{B}(z^*, \delta)$ be defined by Lemma~\ref{lemma:background:FpnonsingCont}. Then there exists $\muhat>0$ and a Lipschitz continuous function $z^{\mu}: (0, \ \muhat] \to  \mathcal{B}(z^*, \delta)$ that satisfies $F_{\mu}(z^{\mu}) = 0$ and
\begin{equation*}
\left\|   z^{\mu} - z^*
 \right\|   \leq C_3 \mu,
\end{equation*} 
where $C_3 = \sup_{z\in \mathcal{B}(z^*, \delta)} \| F'(z) \inv \frac{ \partial F_{\mu} (z)}{\partial \mu} \| $.
\end{lemma}
\begin{proof}
The result follows from the implicit function theorem, see e.g., \cite[p.~128]{OrtRhe00}.
\end{proof}

The following lemma provides a relation between the distance of vectors $z$ to the barrier trajectory and the quantity $\| F_\mu (z)\|$, when the distance is sufficiently small. A corresponding result is also given by Byrd, Liu and Nocedal \cite{ByrLuiNoc98}.
\begin{lemma} \label{lemma:FmuBound}
Under Assumption~\ref{ass1}, let $\mathcal{B}(z^*, \delta)$ and $\muhat$ be defined by Lemma~\ref{lemma:background:FpnonsingCont} and Lemma~\ref{lemma:LipcPath} respectively. For $0<\mu \le \muhat$ and $z$ sufficiently close to $z^{\mu} \in \mathcal{B}(z^*, \delta)$ there exist constants $C_4, C_5 > 0$ such that
\begin{equation*}
C_4  \left\| z- 
z^{\mu} \right\|   \leq \| F_{\mu}(z) \|   \leq C_5 \left\| z- z^\mu \right\| .
\end{equation*} 
\end{lemma}
\begin{proof}
See \cite[p.~43]{ByrLuiNoc98}. 
\end{proof}

The next lemma provides a bound on the Newton direction, $\Delta \hat{z}$, for $z$ sufficiently close to the barrier trajectory.
\begin{lemma} \label{lemma:dzBound}
Under Assumption~\ref{ass1}, let $\mathcal{B}\left( z^*,
  \delta\right)$ and $\muhat$ be defined by
Lemma~\ref{lemma:background:FpnonsingCont} and
Lemma~\ref{lemma:LipcPath} respectively.  For  $0< \mu \leq \muhat$ and $z \in \mathcal{B}(z^*, \delta)$, let $\Delta \hat{z}$ be the solution of (\ref{eq:PDnewtSyst}) with $\mu^+ = \sigma \mu$, where $0< \sigma < 1$. If $z$ is sufficiently close to $z^{\mu}  \in \mathcal{B}(z^*, \delta)$ such that $\| F_{\mu} (z) \| = \mathcal{O}(\mu)$ then
\begin{equation*}
\left\| 
\Delta \hat{z}  \right\|  =  \mathcal{O}(\mu).
\end{equation*}
\end{lemma}
\begin{proof}
Analogous to \cite[Lemma~5]{EkFor2021}.
\end{proof}

\section{A structured modified Newton approach}\label{sec:lowRankUpdMat}
In order to describe the approach and its ideas, we first consider a simple setting with one iteration. For a given $\mu > 0$, consider the interior-point iterate $z^+ \in \mathcal{B}(z^*, \delta)$ defined by  $z^+ = z + \Delta \hat{z}$, where $z \in \mathcal{B}(z^*, \delta)$ and $\Delta \hat{z}$ satisfies (\ref{eq:PDnewtSyst}) with $\mu^+ = \sigma \mu$, $0<\sigma<1$. Since $\Delta \hat{z}$ has been computed with (\ref{eq:PDnewtSyst}) we assume that a factorization of $F'(z)$ is known. Instead of performing another Newton step $\Delta \hat{z}^+$ at $z^+$ for some $\mu^{++} = \sigma^+ \mu^+$, $0 < \sigma^+ \leq 1$, which requires the solution of (\ref{eq:PDnewtSyst}) with $\mu^{++}$ and $z^+$, we would like to compute an approximate solution, which is computationally less expensive, from
\begin{equation} \label{eq:modNewtDz}
B^+ \Delta z^+ = - F_{\mu^{++}}(z^+), \quad \mbox{ where } B^+ = F'(z) + U,
\end{equation}
and $U$ is some low-rank update matrix. A natural question is then how
to choose the update matrix $U$. Gondzio and Sobral~\cite{GonSob2019}
consider rank-1 update matrices such that the distance, in Frobenius norm, between $B^+$ and the previous Jacobian approximation is minimized, when $B^+$ in addition satisfies the secant condition. They show that the sparsity pattern of the first two block rows is maintained, however the sparsity pattern of the third row block row is typically lost.  In contrast, our strategy is, for a given rank restriction $r$ on $U$, to choose $U$ such
that the distance, in both 2-norm and Frobenius norm, between $B^+$ and the actual Jacobian $F'(z^+)$ is minimized. The sparsity of the Jacobian is maintained, however there is no requirement for $B^+$ to fulfill the secant condition.

To further support the choice of update matrix we give some additional theoretical results. First, we show that there is a region where the modified Newton approach produces small errors with respect to the Newton direction. In particular, a region that depends on $\mu$ where the modified Newton direction approaches the Newton direction as $\mu \to 0$. Later, we also discuss general errors, descent directions with respect to our merit function $\phi_\mu(z)$ and conditions for local convergence.

The error of using the modified Jacobian $B^+$ of (\ref{eq:modNewtDz}) is 
 \begin{equation} \label{eq:JacAppErr}
 E^+= F'(z^+) - B^+ = \Delta F'(\Delta \hat{z}) - U.
 \end{equation}
Given a rank restriction $r$, $0\leq r \leq m$, on $U$, the Eckart-Young-Mirsky theorem gives the update matrix $U$ that minimizes the Jacobian error $E^+$, in terms of the measure $\| . \|_2$ and $\| . \|_F$. In Proposition~\ref{prop:U} below, we give an expression for $U$ and show that the resulting modified Jacobian may be viewed as a Jacobian evaluated at a point $\bar{z}^+ = ( \bar{x}^+, \bar{\lambda}^+, \bar{s}^+)$.

\begin{proposition} \label{prop:U}
For $z = (x, \lambda, s)$ and $\Delta z = ( \Delta x, \Delta \lambda,
\Delta s)$, let $F'(z)$ and $\Delta F'( \Delta z)$ be defined by
(\ref{eq:Fp}) and (\ref{eq:dFp}) respectively, and let $z^+=z+\Delta z$. For a given rank $r$, $ 0 \le r \le m$, let $\mathcal{U}_r$ be the set of indices corresponding to the $r$ largest quantities of $\sqrt{(\Delta \lambda_i)^2 + (\Delta s_i)^2}$, $i=1,\dots,m$. The optimal solution of 
\begin{equation*}
\begin{array}{cl} 
\minimize{ U \in \mathbb{R}^{(2m+n)\times (2m+n)}} & {\| F'(z^+) - B^+
  \|} \\
 \subject & B^{+}= F'(z) + U, \\
          & \rank{ ( U )} \le r,
\end{array} 
\end{equation*}
where $\| . \|$ is either of type $\| .\|_2$ or $\| . \|_F$, is
\begin{equation*}
U_* = \sum_{i \in \mathcal{U}_r}  e_{n+m+i} \left(   ( s_i^+ - s_i ) e_{n+i} + (\lambda_i^+ - \lambda_i) e_{m+n+i} \right)^T.
\end{equation*}

In consequence, it holds that
\[B^+ = F'(\bar{z}^+), \mbox{ with }
 (\bar{x}^+_{i}, \bar{\lambda}^+_{i}, \bar{s}^+_{i}) 
= \begin{cases}
(x^{+}_{i}, \lambda^{+}_{i}, s^{+}_{i}) &  i \in \mathcal{U}_r, \\
(x^{+}_{i}, \lambda_i, s_i) & i \in \{1,\dots, m\} \setminus \mathcal{U}_r.
 \end{cases}
\]
\end{proposition}
\begin{proof}
Note that $\| F'(z^+) - B^+ \| = \| \Delta F'(\Delta z) -  U \| = \| E^+\|$ by (\ref{eq:JacAppErr}). The result then follows from the Eckart-Young-Mirsky theorem, stated in Theorem~\ref{thm:EckartYoundMirsky},  together with Lemma~\ref{lemma:SVDdFp}. The last part of the proposition follows directly from performing the update.
\end{proof}

Proposition~\ref{prop:U} shows that each rank-1 term of the sum in $U_*$ added to $F'(z)$ is equivalent to the update of one component-pair, $(\lambda, s)$, in the $\Lambdait$ and $S$ blocks of the Jacobian. The essence is that adding the rank-$r$ update matrix $U_*$ to $F'(z)$ is equivalent to updating pairs $(\lambda_i, s_i)$ to $\left( \lambda^{+}_{i}, s^{+}_{i} \right)$, $i \in \mathcal{U}_r$, and that the modified Jacobian at $z^+$ may be viewed as a Jacobian evaluated at $\bar{z}^+$. In particular, $r=m$ gives $\bar{z}^+ = z^+$ and $B^+ = F'(z^+)$. 

Before we give the analogous result of Proposition~\ref{prop:U} in a more general framework, we show that there exists a region where the modified Newton approach may be started without causing large errors in the search direction. In particular, we give a bound on the search direction error $ \| \Delta \hat{z}^+ - \Delta z^+ \|$, where $\Delta z^+$ satisfies  (\ref{eq:modNewtDz}) with update matrix $U_*$ of rank $r$ as given in Proposition~\ref{prop:U}.  In the derivation, the inverse of $B^+$ is expressed as a Neumann series which requires $\rho ( F'(z^+) \inv E^+) < 1$. We first show that among $U$ such that $\rank(U) \le r$, $U_*$ is sound in regard to the reduction of an upper bound of $\rho (F'(z^+) \inv E^+)$. Thereafter, in Lemma~\ref{lemma:spectralradiusMot} we show that, for iterates $z$ sufficiently close to the barrier trajectory, the quantity $\| F'(z^+) \inv E^+ \|$, and consequently also $\rho (F'(z^+) \inv E^+)$, is bounded above by a constant times $\mu$. This gives the existence of a region, that depends on the barrier parameter, where  $\rho ( F'(z^+) \inv E^+ ) < 1$. 

By assumption $z^+ \in \mathcal{B}(z^*, \delta)$, hence by Lemma~\ref{lemma:background:FpnonsingCont} there exists a constant $M>0$ such that $\| F'(z^+) \inv \|  \le M$, and it holds that
 \begin{equation} \label{eq:spectralRadMot0}
\rho ( F'(z^+) \inv E^+ )  \leq \| F'(z^+) \inv E^+ \|  \leq M  \sigma_{max}(\Delta F'(\Delta \hat{z} ) - U).
 \end{equation}
Lemma~\ref{lemma:SVDdFp} shows that the singular values of $\Delta F'(\Delta \hat{z})$ are given by \linebreak $\sqrt{(\Delta \hat{\lambda}_i)^2 + (\Delta \hat{s}_i)^2}$,  $i=1,\dots,m$. The largest reduction of the upper bound in (\ref{eq:spectralRadMot0}), among $U$ such that $\rank(U) \le r$, is achieved with the rank-$r$ update matrix $U_*$ of Proposition~\ref{prop:U}, which gives 
 \begin{align} \label{eq:spectralRadMot1}
\rho ( F'(z^+) \inv E^+ ) & \leq\| F'(z^+) \inv E^+ \|  \leq M \max_{i=r+1,\dots, m} \sqrt{(\Delta \hat{\lambda}_i)^2 + (\Delta \hat{s}_i)^2} \nonumber \\ 
&= M  \sqrt{(\Delta \hat{\lambda}_{r+1})^2 +( \Delta \hat{s}_{r+1})^2},
 \end{align}
 where the indices $i=1,\dots,m$ are ordered such that $ \sqrt{(\Delta \hat{\lambda}_i)^2 + (\Delta \hat{s}_i)^2}$ are in descending order. Thus supporting the choice of update matrix in regard to the reduction of the upper bound of the spectral radius, and 2-norm, of $F'(z^+)\inv E^+$.

\begin{lemma} \label{lemma:spectralradiusMot}
Under Assumption~\ref{ass1}, let $\mathcal{B}\left( z^*,  \delta\right)$ and $\muhat$ be defined by Lemma~\ref{lemma:background:FpnonsingCont} and Lemma~\ref{lemma:LipcPath} respectively.  For $0<\mu \le \muhat$ and $z \in \mathcal{B}(z^*, \delta)$, define ${z^+= z +  \Delta \hat{z}}$, where $\Delta \hat{z}$ is the solution of (\ref{eq:PDnewtSyst}) with $\mu^+ = \sigma \mu$, $0< \sigma < 1$. Moreover, let ${E^+= \Delta F'(\Delta \hat{z}) -U_*}$ with $U_*$ defined as the rank-$r$, $0 \le r<m$, update matrix $U_*$ of Proposition~\ref{prop:U}. If $z$ is sufficiently close to $z^{\mu}  \in \mathcal{B}(z^*, \delta)$ such that $\| F_{\mu} (z) \| = \mathcal{O}(\mu)$ and $z^+  \in \mathcal{B}(z^*, \delta)$, then
\begin{equation} \label{eq:specRadAsymypBound}
\| F'(z^+) \inv E^+ \|   \le M  C^{(r+1)} \mu,
 \end{equation}
 where $M$ is defined by Lemma~\ref{lemma:background:FpnonsingCont} and $C^{(r+1)} > 0$ is a constant such that \linebreak ${
 \sqrt{(\Delta \hat{\lambda}_{r+1})^2 +( \Delta \hat{s}_{r+1})^2} \le C^{(r+1)} \mu }$ with $ \sqrt{(\Delta \hat{\lambda}_i)^2 + (\Delta \hat{s}_i)^2}$, $i=1,\dots,m$, ordered in descending order. In addition, $C^{(r+1)}$ decreases as $r$ increases.
\end{lemma} 
 \begin{proof}
The point $z$ and direction $\Delta \hat{z}$ satisfy the conditions of Lemma~\ref{lemma:dzBound}, hence there exists a constant $C>0$ such that $\| \Delta \hat{z} \| \leq C \mu$. In consequence, there exist constants $C^{(i)}>0$, $i=1,\dots, m$, such that
 \begin{equation} \label{eq:proof:spectralRadMot}
 \sqrt{(\Delta \hat{\lambda}_i)^2 + (\Delta \hat{s}_i)^2} \le C^{(i)} \mu, \quad i=1,\dots,m.
 \end{equation}
If in addition, $\sqrt{(\Delta \hat{\lambda}_i)^2 + (\Delta \hat{s}_i)^2}$, $i=1,\dots,m$, are ordered in descending order, then $C^{(i)}$,  $i=1,\dots,m$, may be chosen such that  $C^{(1)} \geq ... \geq C^{(m)}$. A combination of (\ref{eq:spectralRadMot1}) and (\ref{eq:proof:spectralRadMot}) gives the result.
\end{proof}

The bound in (\ref{eq:specRadAsymypBound}) of Lemma~\ref{lemma:spectralradiusMot} shows that $\| F'(z^+) \inv E^+ \|$, and by (\ref{eq:spectralRadMot0}) also $\rho ( F'(z) \inv E )$, will be less than unity for sufficiently small $\mu$. Indeed, this is also true when $U$ is a zero matrix, i.e., for a simplified Newton strategy.
The derivation of the result of Lemma~\ref{lemma:spectralradiusMot} utilizes that, for a given rank restriction $r$, $U_*$ of Proposition~\ref{prop:U} is the update matrix that gives the tightest bound in (\ref{eq:spectralRadMot1}). In consequence, $U_*$ is also the rank-$r$ update matrix that gives the tightest upper bound in the result of the lemma, with our analysis. Moreover, $C^{(1)} \geq ... \geq C^{(m)}$, and consequently $C^{(r+1)}$ decreases with increasing $r$.  In addition, (\ref{eq:specRadAsymypBound}) provides an explicit sufficient condition on $\mu$, depending on $M$, or $\| F'(z^+)\inv \|$, and $C^{(r+1)}$, for $\rho ( F'(z^+) \inv E^+ )~<~1$.  

Next we give a bound on the search direction error at $z^+$ with the modified Newton equation (\ref{eq:modNewtDz}) relative to the Newton equation (\ref{eq:PDnewtSyst}) with $\mu^{++}$. It is shown that  the error is bounded by a constant times $\mu^3$ when $\mu^{++}  = \mu^+$ and a constant times $\mu^2$ when $\mu^{++} <  \mu^+$. As may be anticipated, the bound is tighter when  $\mu$ is not decreased in the corresponding iteration.

\begin{theorem} \label{thm:dzErr}
Under Assumption~\ref{ass1}, let $\mathcal{B}\left( z^*,  \delta\right)$, $M$, $C_3$, $\muhat$, be defined by \linebreak Lemma~\ref{lemma:background:FpnonsingCont} and Lemma~\ref{lemma:LipcPath}.  For $0<\mu \le \muhat$, assume that $z \in \mathcal{B}(z^*, \delta)$ is sufficiently close to $z^{\mu}  \in \mathcal{B}(z^*, \delta)$ such that $\| F_{\mu} (z) \| = \mathcal{O}(\mu)$. Define $z^+= z +  \Delta \hat{z}$ where $\Delta \hat{z}$ is the solution of (\ref{eq:PDnewtSyst}) with $\mu^+ = \sigma \mu$, $0< \sigma < 1$. Moreover, let $\Delta z^+$ be defined by  (\ref{eq:modNewtDz}) with $\mu^{++} = \sigma^+ \mu^+$, $0 <\sigma^+ \le 1$, and $U$ as the rank-$r$, $0 \le r<m$, update matrix $U_*$ of Proposition~\ref{prop:U}.  If $z^+ \in \mathcal{B}(z^*, \delta)$, then there exists $\mubar$, $0<\mubar \le \muhat$, such that for $0<\mu \le \mubar$ 
\begin{equation} \label{eq:thmeq}
\left\| \Delta \hat{z}^+ - \Delta z^+ \right \| \leq    \frac{M C^{(r+1)}}{1-M C^{(r+1)}\mu} \big( C_3(1 - \sigma^+) \sigma \mu^2 + \mathcal{O}(\mu^3) \big),
\end{equation}
where $\Delta \hat{z}^+$ is the Newton step at $
  z^+$, given by $F'(z^+)  \Delta 
  \hat{z}^+= -F_{\mu^{++}}(z^+)$, and $C^{(r+1)} > 0$ is such that $
 \sqrt{(\Delta \hat{\lambda}_{r+1})^2 +( \Delta \hat{s}_{r+1})^2} \le C^{(r+1)} \mu $ with $ \sqrt{(\Delta \hat{\lambda}_i)^2 + (\Delta \hat{s}_i)^2}$, $i=1,\dots,m$, ordered in descending order. In addition, $C^{(r+1)}$ decreases as $r$ increases.
\end{theorem}
\begin{proof} See Appendix. \end{proof}

Similarly as for Lemma~\ref{lemma:spectralradiusMot}, the result of Theorem~\ref{thm:dzErr} is also valid for $U$ as a zero matrix. The essence is again that, among update matrices $U$ such that $\rank(U) \leq r$, the rank-$r$ update matrix $U_*$ of Proposition~\ref{prop:U} is the matrix that provides the tightest bound on (\ref{eq:thmeq}) with our analysis. As mentioned, $U_*$ is also the matrix that gives the upper bound in (\ref{eq:specRadAsymypBound}), and in addition, $C^{(r+1)}$ decreases with increasing $r$.  Consequently, the bound in (\ref{eq:thmeq}) decreases with increasing $r$, and larger values of $r$ may thus also increase the region where the result of Theorem~\ref{thm:dzErr} is valid, i.e., the region where the proposed modified Newton approach may be a viable alternative.

\subsection{At a general iteration $k$}\label{sec:genk}
In this section we give a result analogous to Proposition~\ref{prop:U}, at iteration $k$, $k \geq 1$, in a damped modified Newton setting. Consider the sequence $\{z^i\}_{i=0}^k$ generated by $z^{i+1} = z^{i} + \alpha^i \Delta z^i$, $i=0,\dots, k-1$, where $\alpha^i$ is the step size. Suppose that each $\Delta z^i$ satisfies
\begin{equation}  \label{eq:genk:ModNewtDz}
B^{i} \Delta z^{i} = - F_{\mu^i}(z^i), \quad \mbox{ with } B^{i} =  \begin{cases}F'(z^0) & i=0, \\
 B^{i-1} + U^i & i = 1,\dots, k-1,\end{cases}
\end{equation}
for some $\mu^i>0$ and update matrix $U^{i}$ of rank $r^i$ . If at $k=1$, for a given rank $r^k$, the update matrix is chosen as the optimal solution of the optimization problem in Proposition~\ref{prop:U}, then $B^1 = F'(\bar{z}^1)$, for some $\bar{z}^1$. Inductively, at an iteration $k$, $k\ge 1$, for a given $\bar{z}^{k-1}$ and rank $r^k$, $0\le r^k \le m$, we wish to choose $U^k$ as
the optimal solution of
\begin{equation}\label{eq:genk:bestLRprob}
\begin{array}{cl} 
\minimize{ U \in \mathbb{R}^{(2m+n)\times (2m+n)}} & {\| F'(z^k) - B^k
  \|} \\
 \subject & B^{k}= B^{k-1} + U, \quad B^{k-1} = F'(\bar z^{k-1}), \\
          & \rank{ ( U )} \le r^k,
\end{array} 
\end{equation}
where  $\| . \|$ is either of type $\| . \|_2$ or $\| . \|_F$. The optimal solution of (\ref{eq:genk:bestLRprob}), the update of $\bar{z}^k$ from $\bar z^{k-1}$, and the
resulting optimal $B^k$ are shown in
Proposition~\ref{prop:Uk}. This is analogous to the update
from $z^0$ to $\bar z^1$ given in Proposition~\ref{prop:U}. The
essence is that the rank-$r^k$ update 
matrix, defined by the solution of (\ref{eq:genk:bestLRprob}), is
equivalent to updating information corresponding to the $r^k$
largest quantities $\sqrt{ (\lambda^k_i - \bar{\lambda}_i^{k-1})^2 +
  (s^k_i - \bar{s}^{k-1}_i )^2}$. In essence, the $r^k$ largest
deviations from the Newton step are corrected, and
$r^k=m$ gives $B^k=F'(z^k)$.
 \begin{proposition} \label{prop:Uk}
At iteration $k$, $k \geq 1$, for given vectors $z^k$, $\bar z^{k-1}$ and rank $r^k$, $0 \le r^k \le m$, consider optimization problem
(\ref{eq:genk:bestLRprob}). The optimal
solution of (\ref{eq:genk:bestLRprob}) is
\begin{equation*}
U^k_* = \sum_{i \in \mathcal{U}_{r^k}}  e_{n+m+i} \left(   (s^k_i - \bar{s}^{k-1}_i ) e_{n+i} +  (\lambda^k_i - \bar{\lambda}_i^{k-1}) e_{m+n+i} \right)^T,
\end{equation*}
where $\mathcal{U}_{r^k}$ is the set of indices corresponding to the $r^k$ largest quantities of \\$\sqrt{ (\lambda^k_i - \bar{\lambda}_i^{k-1})^2 + (s^k_i - \bar{s}^{k-1}_i )^2}$, $i=1,\dots,m$. In consequence, it holds that 
\begin{equation*}
B^{k} = F'(\bar{z}^{k}), \mbox{ with }  \bar{z}^k =(\bar{x}^k,\bar{\lambda}^k,\bar{s}^k) = \begin{cases}
(x^{k}_{i}, \lambda^{k}_{i}, s^{k}_{i}) &  i \in \mathcal{U}_{r^k}, \\
(x^{k}_{i}, \bar{\lambda}^{k-1}_i, \bar{s}^{k-1}_i) & i \in \{1,\dots, m\} \setminus \mathcal{U}_{r^k}.
 \end{cases}
\end{equation*}
\end{proposition}
\begin{proof}
The proof is analogous to that of Proposition~\ref{prop:U} with $z=\bar{z}^{k-1}$, $z^+= z^k$ and $\Delta z =  z^k - \bar{z}^{k-1}$. 
\end{proof}

In the described approach, $U_*^k$ of Proposition~\ref{prop:Uk} gives $B^k = F'( \bar{z}^k)$ for some $\bar{z}^k$. A direct consequence is that iterates which become primal-dual feasible, i.e., satisfies the first two block equations of (\ref{eq:PDnewtSyst}), will remain so.  Moreover, at iteration $k$, the error of using the modified Jacobian is 
 \begin{equation} \label{eq:Ek}
  E^k  =  F'(z^k) - B^k =  F'(z^k) - (B^{k-1}+ U^k_*)  =  F'(z^k) - F'( \bar{z}^k),
 \end{equation}
and it holds that
  \begin{align*}
\| E^k \|_{2}   \leq \| E^k \|_F =   \sqrt{\sum_{i=1}^m \sum_{j=1}^n \vert E^k_{ij} \vert^2} & =  \sqrt{\sum_{i=1}^m \left( (s_i^k - \bar{s}_i^k)^2 + (\lambda_i^k - \bar{\lambda}_i^k)^2 \right)}  \\ & = \| z^k - \bar{z}^k \|_2,
 \end{align*}
 since $x^k = \bar{x}^k$. Hence it follows that 
 \begin{equation} \label{eq:adPropNorm}
 \| F'(z^k) - F'( \bar{z}^k) \|_{2}  \leq  \| F'(z^k) - F'( \bar{z}^k) \|_{F}  = \| z^k - \bar{z}^k \|_2.
 \end{equation}
In fact, for any $z$ and $\tilde{z}$ it holds that
  \begin{equation*}
 \| F'(z) - F'(\tilde{z}) \|_{2}  \leq  \| F'(z) - F'(\tilde{z}) \|_{F}  \leq \| z -\tilde{z} \|_2,
 \end{equation*}
 which implies that the Lipschitz constant of $F'$ may be chosen as one. Recall that, among the update matrices $U^k$ such that $\rank(U^k) \leq r^k$, $U_*^k$ of Proposition~\ref{prop:Uk} is the update matrix that minimizes $\| F'(z^k) - F'( \bar{z}^k) \|_{F}$. Consequently, by (\ref{eq:adPropNorm}), $U_*^k$ is also the update matrix that minimizes $\| z^k - \bar{z}^k \|_2$. \\
 \\
 Next we show that the update matrix $U_*^k$ of Proposition~\ref{prop:Uk} is sound with respect to reducing an upper bound of the relative error of the search direction $\Delta z^k$.  At iteration $k$, $k\geq1$, $\Delta z^k$ satisfies (\ref{eq:genk:ModNewtDz}) which equivalently can be written as
\begin{equation*}
(F'(z^k) - E^k) ( \Delta \hat{z}^k -  \epsilon^k) = - F_{\mu^k} (z^k), \quad  \epsilon ^k = \Delta \hat{z}^k-\Delta z^k,
\end{equation*}
where $\Delta \hat{z}^k$ satisfies $F'(z^k) \Delta \hat{z}^k = - F_{\mu^k} (z^k)$. Standard perturbation analysis gives
\begin{equation}\label{eq:relErr}
\frac{\| \Delta \hat{z}^k - \Delta z^k  \| }{ \| \Delta z^k \|} \leq \| F'(z^k) \inv \| \| E^k \|,
\end{equation}
Given the restrictions on the update, $U_*^k$ of Proposition~\ref{prop:Uk} is the update that minimizes $\| E^k \|$, and in consequence gives the largest reduction in the upper bound on the relative error (\ref{eq:relErr}).

\subsection{Convergence}
In this section we discuss convergence towards the barrier trajectory,
i.e., convergence of the inner loop of
Algorithm~\ref{alg:simpleIPM}. We first give a condition for descent
direction with respect to our merit function, as our setting is
compatible with linesearch strategies. Thereafter, results are given
in a setting where unit steps are assumed to be tractable. We give
conditions on $B^k$, and thus $r^k$, so that the modified Newton
approach converges locally. The theoretical setting is here widened
slightly from the previous sections, in that we wish to quantify the
effects of the modified Newton approach also for larger values of $\mu$.
For a given $\mu>0$, iterate $z$ and vector $\Delta z$, consider the merit function and the univariate function 
\begin{equation*}
\phi_\mu(z) = \| F_\mu(z) \|, \text{ and } \varphi_\mu(\alpha) =  \| F_\mu (z+\alpha \Delta z) \|, \text{ respectively.}
\end{equation*}
The directional derivative is then
\begin{equation*} 
\nabla \phi_\mu(z)^T \Delta z = \varphi_\mu'(\alpha) \vert_{\alpha=0} =  \frac{d }{d \alpha} \| F_\mu (z+\alpha \Delta z) \| \rvert_{\alpha=0} =  \frac{\Delta z^T F'(z)^T F_\mu (z)}{ \| F_\mu (z) \|}.
\end{equation*}
At iteration $k$, $k\geq 0$, $z^k$ and $\Delta z^k$ in the modified Newton approach satisfies  $F'(\bar{z}^k) \Delta z^k = -F_\mu(z^k)$, for some $\bar{z}^k$. If $F'(\bar{z}^k$) is nonsingular then the directional derivative is
\begin{align} \label{eq:dDirModNewt} 
\nabla \phi_\mu(z^k)^T \Delta z^k & = -  \frac{1}{ \| F_\mu (z^k) \|} F_\mu (z^k)^T F'(z^k) F'(\bar{z}^k) \inv F_\mu (z^k) \nonumber \\
 & =  - \| F_\mu (z^k) \| -   \frac{1}{ \| F_\mu (z^k) \|} F_\mu (z^k)^T E^k F'(\bar{z}^k) \inv F_\mu (z^k),
\end{align}
with $E^k$ as in (\ref{eq:Ek}). From (\ref{eq:dDirModNewt}) it follows that $\Delta z^k$ is a descent direction with respect to $\phi_\mu$ if \begin{equation} \label{eq:dDirCrit}
\| F_\mu (z^k) \|^2 > - F_\mu (z^k)^T E^k F'(\bar{z}^k) \inv F_\mu (z^k). 
\end{equation}
Under the restrictions of the update, the rank-$r^k$ matrix $U^k_*$ of Proposition~\ref{prop:Uk} is chosen such that $\| E^k \|$ is minimized. In addition, $\| E^k \| = 0$ for $r^k = m$. A descent direction can hence always be ensured for $r^k$ sufficiently large. Moreover, in our theoretical setting, Lemma~\ref{lemma:spectralradiusMot} gives the existence of a region, that depends on $\mu$, where the modified Newton approach may be initiated so that $\| E^k \|$ is sufficiently small for (\ref{eq:dDirCrit}) to hold, even when $r^k=0$. However, the essence is again that $U_*^k$ gives the largest reduction of $\| E^k \|$, among $U^k$ such that $\rank(U^k) \leq r^k$.

Next we give a result on local convergence and discuss the modified
Newton approach in the framework of inexact Newton methods. Note that
the local convergence result in the proposition shows balance between
quadratic rate of convergence, which would follow if $B^k=F'(z^k)$,
i.e., $C_6=0$, and $\norm{F'(z^k)\inv}\le C_7$ for all $k$, and linear
rate of convergence, which would follow when $B^k$ differs from
$F'(z^k)$.

\begin{proposition}
  For a given $\mu>0$, assume that $z^\mu$ exists. At an iteration
  $k_0$, consider the sequence of iterates generated by
  $z^{k+1} = z^k + \Delta z^k$, $k = k_0, k_0 +1, \dots,$ where each
  $\Delta z^k$ satisfies (\ref{eq:genk:ModNewtDz}), with $\mu^k = \mu$
  and update matrix of rank-$r^k$, ${0\leq r^k \leq m}$, given by
  $U^k_*$ of Proposition~\ref{prop:Uk}. Assume that at each iteration $k$,
  $(B^k) \inv$ exists and that $r^k$ is chosen such that for all $k$,
  $\| (B^k)\inv \|\| F'(z^k) - B^k \| \le C_6$ for some $C_6<1$. If in
  addition, there is a $C_7$ such that for all $k$,
  $\norm{(B^k)\inv}\le C_7$, then it holds that
\begin{equation}\label{eq:newProp:last}
\| z^{k+1}-z^{\mu} \| \leq \frac{C_7}{2}  \| z^k - z^\mu \|^2 + C_6 \| z^k - z^\mu \|,
\end{equation}
so that $z^k$ converges to $z^\mu$ if
\begin{equation}\label{eq:zconv}
\norm{z^{k_0}-z^\mu} \le \frac{1-C_6}{C_7}.
\end{equation}
\end{proposition}

\begin{proof}
Under the conditions of the proposition, $\Delta z^k$ and $z^k$ satisfy \linebreak ${B^k \Delta z^k = - F_\mu (z^k)}$. At iteration $k+1$ the error may be written as 
\begin{align*}
 z^{k+1}-z^{\mu} & = z^k -  (B^k) \inv F_\mu (z^k) - z^{\mu} \\
 & =     (B^k) \inv \big( F_\mu(z^\mu) -  F_\mu (z^k)-F'(z^k)(z^\mu - z^k) \big) \\
 & \quad - (B^k) \inv \big( B^k -  F'(z^k)  \big) ( z^{\mu} - z^k),
\end{align*}
where $(B^k) \inv F'(z^k) ( z^{\mu} - z^k)$ has been added and subtracted in the second equality. Taking 2-norm while considering Lipschitz continuity of $F'$,  see end of Section~\ref{sec:genk}, and norm inequalities give
\begin{equation}\label{eq:newProof:last}
\| z^{k+1}-z^{\mu} \| \leq \frac{\| (B^k) \inv\| }{2}  \| z^k - z^\mu \|^2 + \| (B^k) \inv \| \| B^k - F'(z^k) \| \| z^k - z^\mu \|.
\end{equation}
Insertion of the assumed $C_6$ and $C_7$ into (\ref{eq:newProof:last})
gives (\ref{eq:newProp:last}). Finally, if \linebreak $\norm{z^{k}-z^\mu} \le
(1-C_6)/C_7$, then (\ref{eq:newProp:last}) gives
\begin{eqnarray*}
\| z^{k+1}-z^{\mu} \| & \leq & \frac{C_7}{2}  \| z^k - z^\mu \|^2 + C_6 \| z^k - z^\mu \|
                               \leq \left(\frac{1-C_6}{2} +C_6\right) \| z^k - z^\mu \| \\
  & \le & \frac{1+C_6}{2} \| z^k - z^\mu \|,
\end{eqnarray*}
so that $z^{k_0}$ satisfying (\ref{eq:zconv}) converges to $z^\mu$, as
$C_6<1$.
\end{proof}

Conditions for local convergence may also be obtained when the modified Newton approach is interpreted in the context of inexact Newton methods \cite{DemEisSte1982}. For a given $\mu >0$, such steps may be viewed on the form
 \begin{equation} \label{eq:inexactNewt}
 F'(z^k) \Delta z^k = - F_{\mu}(z^k) + q^k, \mbox{ where } \| q^k \| / \|  F_{\mu}(z^k)  \| \leq \eta^k.  
 \end{equation} 
The sequence of iterates $z^k + \Delta z^k$ converges to $z^\mu$, with at least linear rate, for $z^0$ sufficiently close to $z^\mu$ if $\eta^k < 1$ uniformly. Given that the iterates converge, the convergence is superlinear if and only if $ \| q^k \| = o(\| F_{\mu}(z^k) \|), \mbox{ as } k \to \infty$.
 
 The modified Newton approach can be put onto the form of
 (\ref{eq:inexactNewt}) under the assumption that $I-E^k
 F'(z^{k})\inv$ is nonsingular, where $E^k$ is given by (\ref{eq:Ek}). Nonsingularity may be ensured with an update matrix $U^k_*$ of Proposition~\ref{prop:Uk} of sufficiently large rank $r^k$, $0\leq r^k \leq m$, or by starting the modified Newton approach when $\mu$ is sufficiently small, as shown by Lemma~\ref{lemma:spectralradiusMot}. A straightforward calculation shows that
 (\ref{eq:genk:ModNewtDz}) at iteration $k$, with $\mu^k = \mu$, can be written as 
\begin{equation} \label{eq:conv1}
  F'(z^k) \Delta z^k  =    - F_{\mu}(z^k) + \big( I - (I-E^k F'(z^k) \inv)\inv \big) F_{\mu}(z^k).
\end{equation}
Identification of terms in (\ref{eq:inexactNewt}) and (\ref{eq:conv1}) gives \[ q^k = \big( I - (I-E^k F'(z^k) \inv)\inv \big) F_{\mu}(z^k).\]
If in addition $\| E^k F' (z^k)\inv  \|  < 1$, then
$q^k = \sum_{j=1}^\infty (E^k F'(z^k)\inv)^j F_{\mu}(z^k)$. Norm inequalities and standard geometric series results give
 \begin{equation*} 
\| q^k \|  \le \frac{ \| E^k F'(z^k) \inv \| }{1 - \| E^k F'(z^k) \inv \| } \|  F_\mu (z^k) \|.
\end{equation*}
Local convergence towards the barrier trajectory follows if, at each iteration $k$, $r^k$ of Proposition~\ref{prop:Uk} is chosen such that $  \| E^k F'(z^k) \inv \| / ( 1 - \| E^k F'(z^k) \inv \| ) < 1$ uniformly. Moreover, the convergence is superlinear if in addition $r^k$  is chosen such that $ \| E^k F'(z^k) \inv \| / ( 1 - \| E^k F'(z^k) \inv \| ) \to 0 $ as $k\to \infty$.
 
Similarly, the modified Newton approach may also be viewed in the framework of inexact interior-point methods in order to study conditions for global convergence, see e.g.,~\cite{Bel1998,ArmBenDus2012,Gon2013}. However, our analysis have only given further technical conditions which are outside the scope of this initial work. Instead, we have chosen to limit our focus to basic supporting results and to the study of practical performance with update matrices of low-rank at larger values of $\mu$.

\subsection{Reduced systems} \label{sec:redRegSolv}
The ideas presented so far have been on the unreduced unsymmetric block 3-by-3 system (\ref{eq:PDnewtSyst}).  In this section we describe the corresponding reduced and condensed system which are similar to those of (\ref{eq:PDnewtSystSym}) and (\ref{eq:PDnewtSystCond}). 

In essence, the system of linear equations to be solved for each iterate $z$ in the modified Newton approach takes the form
 \begin{equation} \label{eq:mNdz}
F'(\bar{z}) \Delta z = - F_\mu (z),
\end{equation}
for some $\bar{z}$. System (\ref{eq:mNdz}) may be reformulated on the reduced form
\begin{equation} \label{eq:PDmodNewtSystSym}
\begin{pmatrix}
H & A^T \\ A & -\bar{\Lambdait} \inv \bar{S} \end{pmatrix} \begin{pmatrix}  \Delta x \\ -\Delta \lambda \end{pmatrix}  = - \begin{pmatrix} Hx +c - A^T \lambda \\ Ax -b -s + \bar{\Lambdait} \inv ( \Lambdait S e - \mu e ) \end{pmatrix},
\end{equation} 
together with $\Delta s = - \bar{\Lambdait} \inv ( \Lambdait S e - \mu e ) - \bar{\Lambdait} \inv \bar{S} \Delta \lambda$. Schur complement reduction of $\bar{\Lambdait} \inv \bar{S}$ in (\ref{eq:PDmodNewtSystSym}) gives the condensed form
\begin{align}  \label{eq:PDmodNewtSystCond}
(H + A^T \bar{S} \inv \bar{\Lambdait} A ) \Delta x  = & - (Hx + c - A^T \lambda ) \nonumber \\ 
&  - A^T \bar{S} \inv ( \bar{\Lambdait} (Ax -b-s) + \Lambdait Se - \mu e),
\end{align}
together with $\Delta \lambda = - \bar{S} \inv ( \bar{\Lambdait} (Ax-b-s) + \Lambdait S e - \mu e ) - \bar{S} \inv \bar{\Lambdait} A \Delta x$.
As mentioned, the proposed rank-$r$ update matrix of Proposition~\ref{prop:Uk} is equivalent to updating $r$ component
pairs $(\bar{\lambda}, \bar{s})$. The change between iterations in the matrices of (\ref{eq:PDmodNewtSystSym}) and (\ref{eq:PDmodNewtSystCond}) is thus of rank $r$. In consequence, low-rank updates on the factorization of the matrix of (\ref{eq:PDmodNewtSystSym}), or (\ref{eq:PDmodNewtSystCond}), may also be considered. 

\subsection{Compatibility with previous work on interior-point methods}
In order to have simple notation, we have chosen to formulate our problem on the form (\ref{eq:IQP}), with inequality constraints only. Analogous results hold for quadratic programs on standard form, as considered in
\cite{GonSob2019,AltGon1999,MorSimTan2016,FriOrb2012,MorSim2017}. However, when working with reduced systems, then the update will be on the diagonal of the $H$-matrix in the symmetric
block 2-by-2 indefinite system.

Moreover, the proposed approach is also compatible with regularized methods for quadratic programming, e.g., \cite{FriOrb2012,AltGon1999,SauTom1996}, as long as the scaling of the regularization is not changed at iterations where the modified Jacobian is updated by a low-rank matrix. The scaling of the regularization may be changed at a refactorization step, e.g., on the form suggested in (\ref{eq:refact}) of Section~\ref{sec:impl}.  

As each modified Jacobian may be viewed as a Jacobian evaluated at a different point, the modified Newton approach may also be interpreted in the framework of previous work on stability, effects of
finite-precision arithmetic and spectral properties of the arising systems,
e.g., \cite{Wri2001,Wri1999,Wri1997,Wri1995s,ForGilShi1996,MorSimTan2017,GreMouOrb2014,ApuSimSer2010,MorSimTan2016,FGG07,MorSim2017}.

\section{Implementation} \label{sec:impl}
All numerical simulations were performed in \texttt{matlab} on
benchmark problems from the repository of convex quadratic programming
problems by Maros and M\'{e}sz\'{a}ros \cite{MarMes1999}. Many of
these problems contain both linear equality and linear inequality
constraints. However, in order not to complicate the description of the implementation with further technical details, we choose to give the description for problems on the same form as in previous sections. Note however that some of the parameters will depend on quantities related to the format of benchmark problems. 

\subsection{Basic algorithm} \label{sec:basicMethods}
The aim is to study the fundamental behavior of the modified Newton approach as
primal-dual interior-point methods converge. In particular, when each search direction is generated with a modified Newton equation on the form (\ref{eq:genk:ModNewtDz}), with update matrix $U^k_*$ of Proposition~\ref{prop:Uk}, relative to a Newton equation on the form (\ref{eq:PDnewtSyst}). In order not to risk combining effects of the proposed approach with effects from other features in more advanced methods, we chose to implement the
modified Newton approach in a simple interior-point
framework. Moreover, all systems of linear equations were solved with \texttt{matlab}'s built in direct solver. No low-rank update of factorizations was used or implemented for the numerical tests.  In consequence, the results are not dependent on the particular factorization or the procedure used to update the factors. For low-rank updates of matrix factorizations, see e.g.,~\cite{GilMurSau1987,Den2010,StaGriBol2007,GilGolMurSau1974}.
Our basic interior-point algorithm is similar to Algorithm~19.1 in Nocedal and Wright \cite[Ch.~19, p.~567]{NocWri06}, however, termination and the update of $\mu$ are based on the merit function $\phi_\mu (z) = \| F_\mu (z) \|$.

\begin{algorithm}[H]
\caption{Basic interior-point method  (\ref{eq:IQP}).}
\begin{algorithmic}[1]
\State $k \gets 0, \> \>$    $\mu^k \gets$ initial $\mu$,  \quad
\State $(x^k, \lambda^k, s^k) \gets$ Point such that $\lambda^k >0, s^k >0$ and $\| F_{\mu^k/\sigma}(x^k, \lambda^k, s^k) \|  < \mu^k/\sigma$.
\State \textbf{while} $\| F_0 (z^k) \| > \varepsilon^{tol}$ \textbf{do}
\State\hspace{\algorithmicindent} \textbf{while} $\| F_{\mu^k} (z^k) \| > \mu^k$ \textbf{do}
\State\hspace{\algorithmicindent}\hspace{\algorithmicindent} $(\Delta x^k, \Delta \lambda^k, \Delta s^k)   \quad \>  \> \> \>  \gets$ search direction
\State\hspace{\algorithmicindent}\hspace{\algorithmicindent} $(\alpha_P^k, \alpha_D^k)  \qquad \quad  \qquad \> \gets \big( \min \{ 1, 0.98\alpha_P^{max,k}\}, \min \{ 1,0.98\alpha_D^{max,k}\} \big)$ 
\State\hspace{\algorithmicindent}\hspace{\algorithmicindent} $(x^{k+1}, \lambda^{k+1}, s^{k+1})  \quad \> \>  \gets (x^k + \alpha_P^k \Delta x^k, \lambda^k + \alpha_D^k \Delta \lambda^k, s^k +\alpha_P^k \Delta s^k )$ 
\State\hspace{\algorithmicindent}\hspace{\algorithmicindent}   $\mu^{k+1} \gets \mu^k$, \quad $k \gets k+1$
\State\hspace{\algorithmicindent} \textbf{end while}
\State\hspace{\algorithmicindent}  $\mu^{k+1} \gets \sigma \mu^k$
\State  \textbf{end while}

\end{algorithmic}
\label{alg:simpleIPM}
\end{algorithm} 
\noindent In Algorithm~\ref{alg:simpleIPM} at iteration $k$, $\alpha_P^{max,k}$ and $\alpha_D^{max,k}$ are the maximum feasible step sizes for $s^k$ along $\Delta s^k$ and $\lambda^k$ along $\Delta \lambda^k$ respectively.

Our reference method, which in all experiments is denoted by \texttt{Newton}, is defined by Algorithm~\ref{alg:simpleIPM}
where the search direction at iteration $k$, $\Delta z^k = (\Delta x^k, \Delta \lambda^k, \Delta s^k)$, satisfies
(\ref{eq:PDnewtSyst}). The method, whose behavior we aim to study, is defined by Algorithm~\ref{alg:simpleIPM} where the search direction satisfies (\ref{eq:genk:ModNewtDz}), with an update matrix of rank $r^k= r$ given by
$U^k_*$ of Proposition~\ref{prop:Uk}. This method is in the numerical experiments denoted by  \texttt{mN-r($r$)}. 
Although the rank of the update matrices can be varied between the iterations, this initial
study is limited to update matrices of constant rank in order to keep the comparisons clean.

\subsection{Benchmark problems} \label{sec:benchmarkProblems}
Each problem was pre-processed and put on an equivalent form with $n$ $x$-variables, $m_{in}$ inequality constraints and $m_{eq}$ equality constraints. 
The total number of variables in the primal-dual formulation is thus $N = n + m_{eq}+ 2 m_{in}$ variables, see Appendix for a description and formulation of the systems that arise.
A trivial equality constraint that fixed a variable at any of its bounds was removed from the problem along with the variable. A problem was accepted if $m_{in} \geq 4$ and, in addition, if \texttt{Newton} converged from a given initial solution. Due to the simplicity of \texttt{Newton} convergence was not achieved for some problems due to reasons as, non-trivial equality constraints fixing variables at its bounds, singular Jacobians caused by linearly dependent equality constraints, etc.
Moreover, we were not able to run \texttt{CONT-300}, \texttt{BOYD1} and \texttt{BOYD2} due to memory restrictions.
These conditions reduced the benchmark set, $\mathcal{P}$, to 90 problems (out of 138).
The problems were divided into the three subsets: small, $\mathcal{S}$, medium, $\mathcal{M}$, and large, $\mathcal{L}$. The sets were defined as follows: $\mathcal{S} = \{ p \in \mathcal{P}: N < 500 \}$, $\mathcal{M} = \{ p \in \mathcal{P}:  500 \leq N < 10000 \}$ and $\mathcal{L} = \{ p \in \mathcal{P}: N \geq 10000 \}$. Consequently, $\vert\mathcal{S}\vert = 25$, $\vert\mathcal{M}\vert = 37$ and $\vert\mathcal{L}\vert = 28$. The specific problems of each group and details on their individual sizes can be found in Appendix.

\subsection{Heuristics} \label{sec:heuristics}
The theoretical results in Section~\ref{sec:lowRankUpdMat} concern iterates sufficiently close to the trajectory for sufficiently small $\mu$, or when the rank of each update matrix is sufficiently large. However, we are also interested in studying the behavior of the modified Newton approach beyond this setting. In particular, for larger vales of $\mu$ and when each update matrix is of low-rank. In this section, we discuss the behavior in these cases. In order to improve performance we also suggest two heuristics and a refactorization strategy. In essence, the heuristics allow for change of indices in the set $\mathcal{U}_{r^k}$ of Proposition~\ref{prop:Uk}, and the refactorization strategy limits the total rank change on an initial Jacobian. 

Numerical experiments with \texttt{mN-r($r$)}, for small $r$, have shown that convergence may slow down due to small step sizes $\alpha^k_P$ and $\alpha^k_D$. Small step sizes can be caused by few components in the modified Newton direction which differ considerably from those in the Newton direction. We first show some numerical evidence of this behavior and suggest a partial explanation on which we base two heuristics. The effectiveness of each heuristic is then illustrated, and finally a refactorization strategy is included in the modified Newton approach. Step sizes and convergence, in terms of the measure $\| F_\mu \|$ with $\mu = 0$, for \texttt{Newton} and  \texttt{mN-r($r$)} with  $r =[0,\> 2, \> 4]$ are shown in the left-hand side of Fig.~\ref{fig:stepSizes}. The results are for benchmark problem \texttt{qafiro} with parameters $\mu^0 = 1$, $\sigma = 0.1$ and $\varepsilon^{tol} = 10^{-6}$. The right-hand side of the figure shows the inverse of the limiting step sizes and the relative error in the search direction at the iteration marked by the red circle of \texttt{mN-r(2)}, hence large spikes imply small step sizes. Moreover, the figure only contains negative components of the modified Newton direction. The result for $r= 0$, i.e., simplified Newton, is given to illustrate that low-rank updates can indeed make a difference compared to a simplified Newton approach for which some of our theoretical results are still valid, although in a smaller region.
\begin{figure}[H]
\centering
\parbox{6.2cm}{
\includegraphics[width=6.2cm]{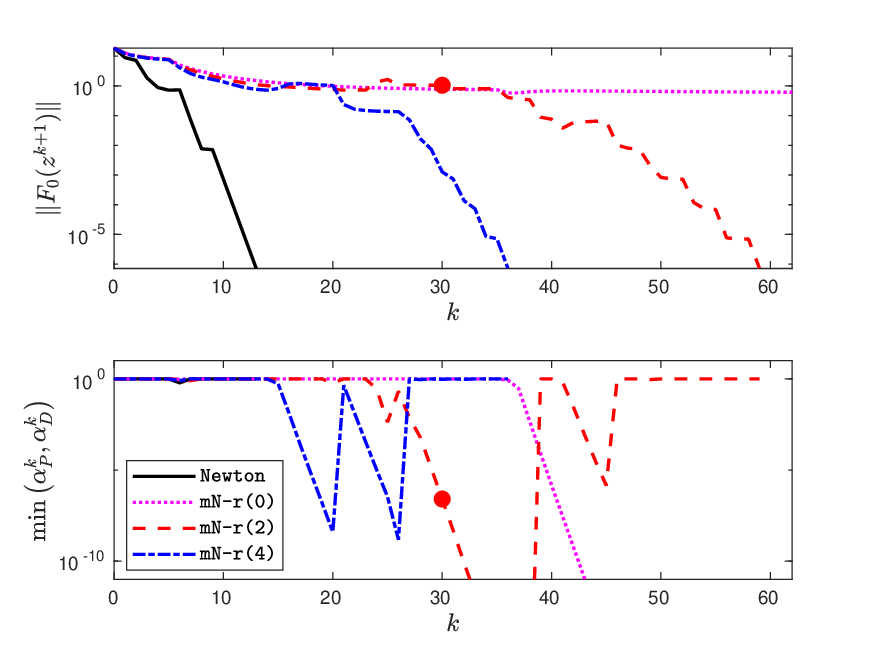}}
\hspace{-8mm}
\begin{minipage}{6.2cm}
\includegraphics[width=6.2cm]{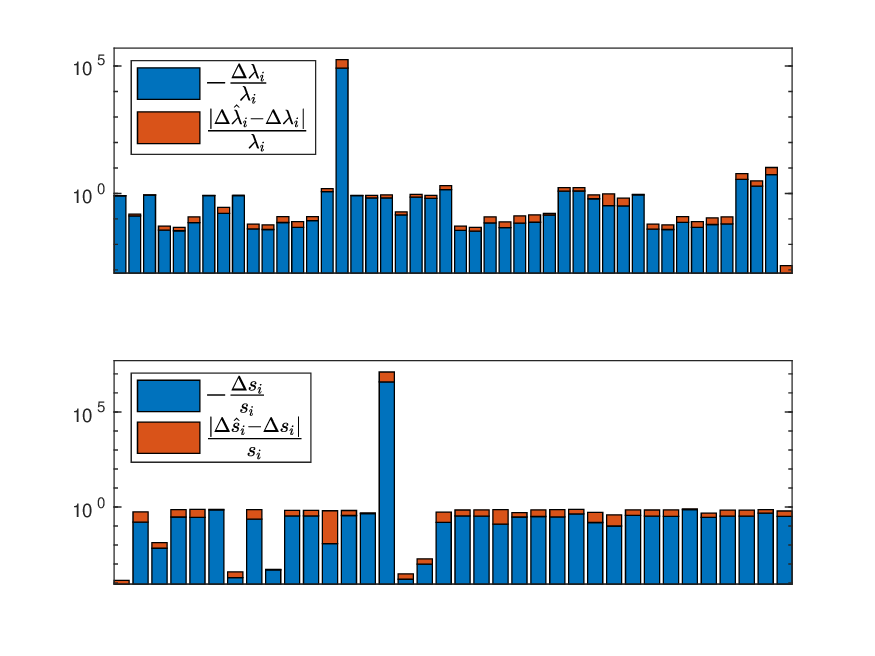} 
\end{minipage}
\caption{The left-hand side shows step sizes and convergence on benchmark problem \texttt{qafiro}. The right-hand side shows the inverse of the limiting step sizes and the relative error in the search direction for negative components of the modified Newton direction in \texttt{mN-r(2)}, at the iteration marked by the red circle. }
\label{fig:stepSizes}
\end{figure}
The results in the left-hand side of Fig.~\ref{fig:stepSizes} indicate that convergence may slow down with the low-rank modified Newton approach due to small step sizes. The right-hand side of Fig.~\ref{fig:stepSizes} suggests that small steps may be caused by large relative errors in certain components of the search direction. The results are similar to those shown by Gondzio and Sobral~\cite{GonSob2019} for quasi-Newton approaches, hence indicating that the proposed modified Newton approach suffers from the same phenomenon as quasi-Newton approaches. In theory, zero steps are not harmful for the modified Newton approach, as long as the Newton step makes progress from this point, since after $m/r$ iterations with zero steps the modified Jacobian will indeed be the Jacobian at that point. In practice however, close to zero steps have negative effects on the convergence. In consequence we would like to understand what causes these steps and how to avoid them. 

The partial solution $\Delta x$ of (\ref{eq:mNdz}) satisfies (\ref{eq:PDmodNewtSystCond}). For $z$ sufficiently close to $z^*$,  (\ref{eq:PDmodNewtSystCond}) may be approximated by
\begin{align} \label{eq:heurMotEq1}  
(H + A_\mathcal{A}^T \bar{S}_\mathcal{A} \inv \bar{\Lambdait}_\mathcal{A} A_\mathcal{A} +A_\mathcal{I}^T  \bar{S}_\mathcal{I} \inv \bar{\Lambdait}_\mathcal{I}  A_\mathcal{I} )  \Delta x & = - A_\mathcal{A}^T \bar{S}_\mathcal{A} \inv ( \Lambdait_\mathcal{A} S_\mathcal{A} e - \mu e) \nonumber \\
& \quad - A_\mathcal{I}^T \bar{S}_\mathcal{I} \inv ( \Lambdait_\mathcal{I} S_\mathcal{I} e - \mu e),
\end{align}
where $\mathcal{A}$ and $\mathcal{I}$ are sets of indices corresponding to active and inactive constraints at the solution $z^*$ respectively, i.e., $\mathcal{A} = \{ i : s_i^* = 0, \  i=1, \dots, m. \}$ and $\mathcal{I} = \{ i : \lambda_i^* = 0, \  i=1, \dots, m. \}$.
If the modified Newton approach is initiated for small $\mu$, or if the rank of each update matrix is sufficiently large, then $\bar{z} \approx z$. If in addition, $A_\mathcal{A} \Delta x$ is sufficiently large, i.e, $\Delta x$ is not in or almost in the null-space of $A_\mathcal{A}$ (if it is then the the search direction will not cause limiting steps), then the dominating terms of (\ref{eq:heurMotEq1}) are
\begin{equation} \label{eq:heurMotEq2}  
A_\mathcal{A}^T \bar{S}_\mathcal{A} \inv \bar{\Lambdait}_\mathcal{A} A_\mathcal{A}  \Delta x = - A_\mathcal{A}^T \bar{S}_\mathcal{A} \inv ( \Lambdait_\mathcal{A} S_\mathcal{A} e - \mu e).
\end{equation}
In essence, (\ref{eq:heurMotEq2}) is an approximation of (\ref{eq:heurMotEq1}) in the particular case. 
By our assumptions $A_\mathcal{A}$ has full row rank. Consequently, component-wise (\ref{eq:heurMotEq2}) gives
\begin{equation} \label{eq:cmpStepSizeAnalysis1}
\frac{\bar{\lambda}_i}{\bar{s}_i} A_i \Delta x =  \frac{ \mu - \lambda_i s_i}{\bar{s}_i}, \qquad i \in  \mathcal{A},
\end{equation}
where $ A_i$ denotes the $i$th row of $A$.
  Equation (\ref{eq:cmpStepSizeAnalysis1}) gives an approximate description of how each pair $(\bar{\lambda}_i,\bar{s}_i)$, $i\in\mathcal{A}$ affects $A_i \Delta x$, the inner product of the search direction $\Delta x$ and the corresponding
  constraint $A_i$. This means that each pair $(\bar{\lambda}_i,\bar{s}_i)$, $i\in\mathcal{A}$, affects the angle between $\Delta x$ and constraint $A_i$, and/or $ \| \Delta x \|$.  Both errors in angle and large $\| \Delta x \|$ may cause small step sizes.   In the proposed modified
  Newton approach, depending on the rank of the update matrix, some of
  the factors in the modified Jacobian, i.e., some components
  $\bar{\lambda}_i / \bar{s}_i$, $i\in \mathcal{A}$, may contain
  information from previous iterates. The analysis above suggests that
  it may be important to update certain components-pairs $(\lambda,s)$
  in order to avoid limiting steps. Such pairs may not be updated with
  the matrix of Proposition~\ref{prop:U} or Proposition~\ref{prop:Uk}.

    In light of the discussion above and the results of
    Fig.~\ref{fig:stepSizes}, we construct two heuristics in an
    attempt to decrease negative effects on convergence caused by
    small step sizes. Both heuristics have an update matrix $U^k$ analogous
    to the one given by Proposition~\ref{prop:Uk}, with
\begin{equation}\label{eq:Uheuristic}
U^k = \sum_{i \in \mathcal{U}^k}  e_{n+m+i} \left(   (s^k_i - \bar{s}^{k-1}_i ) e_{n+i} +  (\lambda^k_i - \bar{\lambda}_i^{k-1}) e_{m+n+i} \right)^T,
\end{equation}
where $\mathcal{U}^k$ is an index set of cardinality $r$. However,
not all indices in $\mathcal{U}^k$ are chosen according to the
criteria of Proposition~\ref{prop:Uk}, so that $\mathcal{U}^k$ of the
heuristics may differ from $\mathcal{U}_r$ of Proposition~\ref{prop:Uk}. (Note that $r^k = r$, for all $k$, in the numerical tests.)
The first 
  heuristic can have at most two indices that differ between
  $\mathcal{U}^k$ and $\mathcal{U}_r$, whereas 
  the second is more flexible and can potentially change all $r$ indices.
      We choose to replace indices instead of adding indices in order to
  obtain a fair comparison in the study of the heuristics.

\subsubsection*{Heuristic \texttt{H1}}
The idea of the first heuristic is to ensure that information
corresponding to component-pairs $(\lambda,s)$ is updated if either
limited the step size in the previous iteration. At iteration $k$, $k
\ge 1$, $\mathcal{U}^k$ is based on $\mathcal{U}_{r}$ of
  Proposition~\ref{prop:Uk}, but the last one or two indices are replaced by
\begin{equation*}
\hat{i}_1 = \argmin_{i: \Delta \lambda^{k-1}_{i} < 0} \frac{\lambda^{k-1}_{i}}{-\Delta \lambda^{k-1}_{i}}, \quad  \mbox{ and } \quad
\hat{i}_2 = \argmin_{i: \Delta s^{k-1}_{i} < 0} \frac{s^{k-1}_{i}}{-\Delta s^{k-1}_{i}},
\end{equation*}
if $\min_{i: \Delta \lambda^{k-1}_{i} < 0} \frac{\lambda^{k-1}_{i}}{-\Delta \lambda^{k-1}_{i}} <1 \land \hat{i}_1 \notin \mathcal{U}_{r}$ and/or $\min_{i: \Delta s^{k-1}_{i} < 0} \frac{s^{k-1}_{i}}{-\Delta s^{k-1}_{i}} <1 \land \hat{i}_2 \notin \mathcal{U}_{r}$ respectively.
\subsubsection*{Heuristic \texttt{H2}}
The principle of the second heuristic is based on the observation in
the analysis above. Similarly as in \texttt{H1} the idea is to ensure
that certain component-pairs $(\lambda, s)$ are updated. In
particular, the components with the largest relative error of the
ratio in the left-hand of (\ref{eq:cmpStepSizeAnalysis1}). However,
the set of active constraints at the solution is unknown, instead all
components which could have limited the step size in the previous
iteration are considered in the selection. At iteration $k$, $k \ge
1$, $\mathcal{U}^k$ is based on $\mathcal{U}_{r}$ of
  Proposition~\ref{prop:Uk}, but at most $r$ indices are replaced by
  the indices corresponding to the, at most, $r$ largest quantities of
\begin{equation*}
\frac{\vert \lambda^{k}_{i}  /  s^{k}_{i} -  \bar{\lambda}^{k}_{i} /   \bar{s}^{k}_{i}\vert}{\lambda^{k}_{i} / s^{k}_{i} }, \quad i \in \mathcal{H}^k
\end{equation*} 
where $\mathcal{H}^k = \{ i: \Delta \lambda^{k-1}_{i} <0 \land \frac{\lambda^{k-1}_{i}}{-\Delta \lambda^{k-1}_{i}} < 1 \} \cup \{ i: \Delta s^{k-1}_{i} <0 \land \frac{s^{k-1}_{i}}{-\Delta s^{k-1}_{i}} < 1 \}$. 

\subsubsection*{Heuristic test}
To demonstrate the impact of heuristic \texttt{H1} and \texttt{H2} we
show results in Fig.~\ref{fig:stepSizesWheur} which are analogous to
those in the left-hand side of Fig.~\ref{fig:stepSizes}. The methods \texttt{mN-r($r$)-H1} and \texttt{mN-r($r$)-H2} denote \texttt{mN-r($r$)}, $r = [2, 4]$, combined with heuristic \texttt{H1} and \texttt{H2} respectively. In addition, Table~\ref{table:aveStep} shows the average of the sum $(\alpha^k_P + \alpha^k_D)/2$ for a subset of the benchmark problems. The subset contains problems from each of the sets $\mathcal{S}$, $\mathcal{M}$ and $\mathcal{L}$.
\begin{figure}[H]
\centering
\parbox{6.2cm}{
\includegraphics[width=6.2cm]{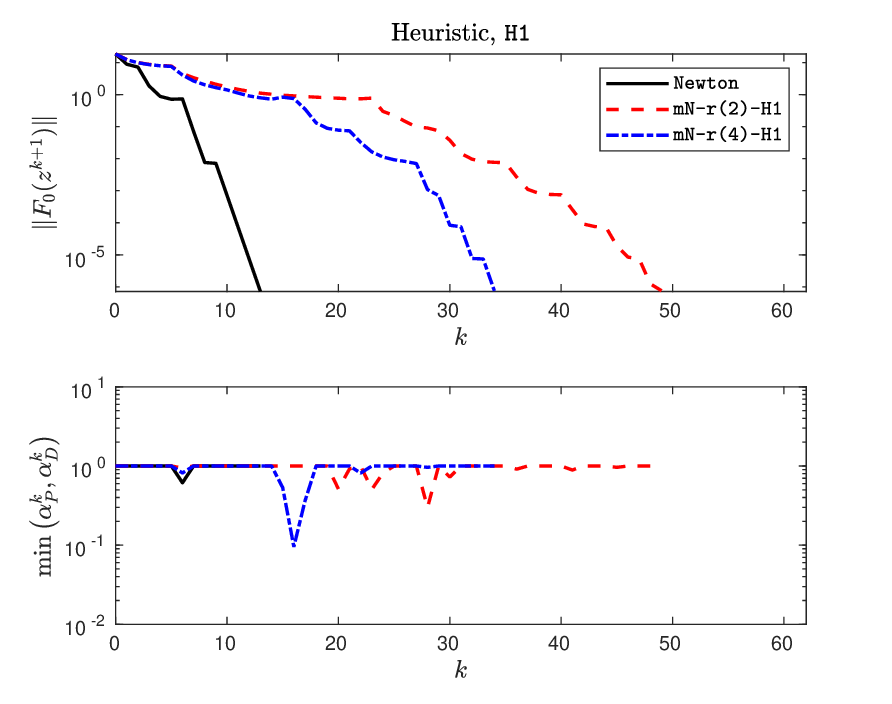}}
\hspace{-8mm}
\begin{minipage}{6.2cm}
\includegraphics[width=6.2cm]{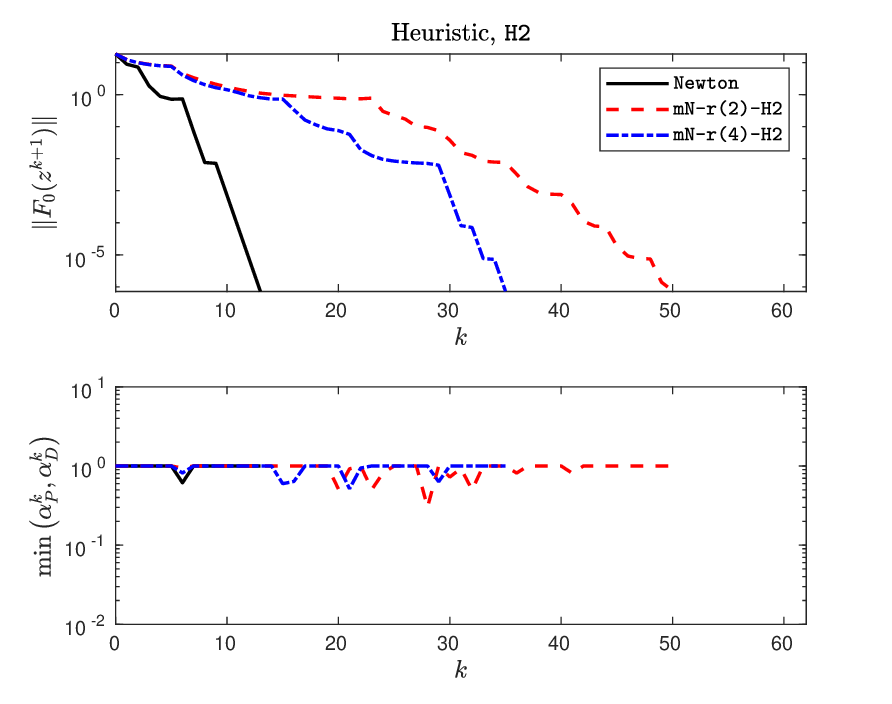} 
\end{minipage}
\caption{Step sizes and convergence for \texttt{mN-r($r$)},  $r = [2, 4]$, combined with heuristic \texttt{H1} and \texttt{H2} on benchmark problem \texttt{qafiro}.}
\label{fig:stepSizesWheur}
\end{figure}
The results of Fig.~\ref{fig:stepSizesWheur} show that \texttt{mN-r($r$)-H1} and \texttt{mN-r($r$)-H2}, $r = [2, 4]$, use larger step sizes, and converges in fewer iterations, compared to \texttt{mN-r($r$)} in Fig.~\ref{fig:stepSizes}. Hence showing that the heuristics \texttt{H1} and \texttt{H2} have the intended effect on benchmark problem \texttt{qafiro}.  
\begin{table}[H]
\caption{Average of the sum $(\alpha^k_P + \alpha^k_D)/2$ for a subset of the benchmark problems.}
\begin{footnotesize}
\begin{tabular}{l|cccccccccc} \label{table:aveStep}
 &  & \texttt{Newton} &  & \multicolumn{3}{c}{\texttt{mN-r(2)}} &  & \multicolumn{3}{c}{\texttt{mN-r(4)}} \\
 &  &        &  &          & \texttt{H1}       & \texttt{H2}      &  &        & \texttt{H1}       & \texttt{H2}      \\ \cline{3-3} \cline{5-7} \cline{9-11} 
\texttt{hs118}    &  & 0.964  & &  0.370  &  \textbf{0.844}  &  0.814  & &  0.513  &  0.817  &  \textbf{0.818}  \\
\texttt{qafiro}   &  & 0.986  & &  0.762  &  \textbf{0.973}  &  0.966  & &  0.846  &  0.957  &  \textbf{0.971}  \\
\texttt{primal1}  &  & 0.965  & &  0.677  &  0.945  &  \textbf{0.956}  & &  0.817  &  0.936  &  \textbf{0.957}  \\
\texttt{dualc8}   &  & 1.000  & &  0.584  &  0.909  &  \textbf{0.911}  & &  0.845  &  0.875  &  \textbf{0.916}  \\
\texttt{laser}    &  & 0.994  & &  0.072  &  \textbf{0.883}  &  0.723  & &   0.111  &  0.865  &  \textbf{0.890}  \\
\texttt{yao}      &  & 0.958  & &  0.683 &   \textbf{0.859}  &  0.444  & & 0.732   &   \textbf{0.753}  &  0.462  \\
\texttt{stcqp2}   &  & 0.999  & &  0.705 &   \textbf{0.971}  &  0.916  & & 0.699  &   0.957  &   \textbf{0.976}  \\
\texttt{ubh1}  &  & 1.000  & &  0.695  &  0.995  &  \textbf{0.998}  & &  0.699  &   0.995 &    \textbf{0.999}\\
\texttt{aug2dqp}  &  & 0.996  & &  0.832 &   0.952  &  \textbf{0.970}  & & 0.589  &  \textbf{0.997}  &  0.994  
\end{tabular}
\end{footnotesize}
\end{table}
The results in Table~\ref{table:aveStep} indicate that \texttt{H1} and
\texttt{H2} have the intended effect on more benchmark problems but
also that they are not effective on all problems. Part of the reason
is that the rank of each update matrix is restricted to
2 and 4 respectively, but for some problems there are many components which limit
the step size. For instance, \texttt{H1} can at most take two of these
and \texttt{H2} can, at most, take as many as the maximum rank of the
update. The results for problem \texttt{yao} is an example where
\texttt{H2} does worse than without heuristic. The heuristic replaced
indices in $\mathcal{U}_r$ which caused low quality in the search
direction, indicating that it may be beneficial to also update the
information that is suggested in Proposition~\ref{prop:Uk}. Numerical
experiments have further shown small step sizes can be avoided by
allowing update matrices of varying rank. In particular, update matrices where the rank
is determined by the components which potentially limit the step
size. However, numerical experiments have also shown that avoiding
small step sizes is not sufficient to obtain increased convergence speed. Similarly as in the results of Gondizo and Sobral for quasi-Newton approaches \cite{GonSob2019}, numerical experiments have shown that it is occasionally important to use the Jacobian at $z^k$ instead of $\bar z^k$  to improve convergence.  In light of this, we will limit the number of allowed steps before the modified Jacobian is refactorized. In consequence the total rank change of an actual Jacobian will be limited. In the following numerical simulations, the modified Newton approaches \texttt{mN-r($r$)}, \texttt{mN-r($r$)-H1} and \texttt{mN-r($r$)-H2}  include a refactorization strategy on the form
\begin{equation} \label{eq:refact}
B^k = \begin{cases} F'(z^k) & k=0, l+1, 2l+2, 3l+3, \dots, \\
B^{k-1} + U^k & k \neq 0, l+1, 2l+2, 3l+3, \dots,
\end{cases}
\end{equation} 
where $U^k$ is given by (\ref{eq:Uheuristic}) for index
sets $\mathcal{U}^k$ corresponding to \texttt{H1}, \texttt{H2}
for  \texttt{mN-r($r$)-H1}, \texttt{mN-r($r$)-H2},
 and $\mathcal{U}_r$ of Propositon~\ref{prop:Uk} for \texttt{mN-r($r$)}, respectively.

In general, other refactorization strategies or dynamical procedures may be considered, e.g.,~instead of accepting all steps, refactor or increase the rank of the update matrix if a particular step is deemed bad for some reason. 

\section{Numerical results} \label{sec:numRes}
In this section we give results on the form of number of iterations and factorizations. The results are meant to give an initial indication of the performance of the proposed modified Newton approach in a basic interior-point framework. The results are for the methods \texttt{Newton}, \texttt{mN-r($r$)}, \texttt{mN-r($r$)-H1} and \texttt{mN-r($r$)-H2}, with $r =[2, \> 16]$, described in Section~\ref{sec:impl}. In essence the methods differ in how the search direction is computed. The direction at iteration $k$ satisfies (\ref{eq:PDnewtSyst}) in \texttt{Newton} and (\ref{eq:genk:ModNewtDz}) in the \texttt{mN}-methods. In contrast to Section~\ref{sec:impl}, here the \texttt{mN}-methods also include the refactorization strategy described in (\ref{eq:refact}). Due to the large variety in number of inequality constraints and number of variables in each benchmark problem, the parameter $l$ of (\ref{eq:refact}) was defined as the closest integer to $l_{\mathcal{S}}$,  $l_{\mathcal{M}}$ and $l_{\mathcal{L}} $ for $p \in \mathcal{S}$, $p \in \mathcal{M}$ and $p \in \mathcal{L}$ respectively, see Table~\ref{table:l} for the specific values. The computational cost of a refactorization of the unreduced, reduced and condensed system all depends on the sparsity structure given by the specific problem. We therefore choose $l_{\mathcal{S}}$,  $l_{\mathcal{M}}$ and $l_{\mathcal{L}}$  such that they relate to the full rank change that corresponds to a new factorization. The values of Table~\ref{table:l} were chosen such that a low-rank update is performed as long as the total rank change on an actual Jacobian is not larger than a factor of $1/2$, $1/10$ and $1/100$ for the small, medium and large problems respectively. 
Moreover, the parameters of Algorithm~\ref{alg:simpleIPM} were chosen as follows: $\sigma = 0.1$, termination tolerance $\varepsilon^{tol} = 10^{-6}$ for the small and medium sized problems and $\varepsilon^{tol} = 10^{-5}$ for the large sized problems. In each run, the initial $(x^0, \lambda^0, s^0)$ was found with \texttt{Newton}, with stopping criteria corresponding to the requirement on the initial solution in Algorithm~\ref{alg:simpleIPM}. 
\begin{table}[H]
\begin{center} 
\caption{Refactorization parameter for the different problem sizes, $m_{in}$ is the number of inequality constraints.}  \label{table:l}
\begin{tabular}{c|c|c}
 $l_{\mathcal{S}}$ &  $l_{\mathcal{M}}$ &  $l_{\mathcal{L}}$ \\ \hline
$   m_{in} r /2$ & $   m_{in} r/10$ & $  m_{in} r/100$ 
\end{tabular}
\end{center}
\end{table}
Results are first shown for problems in the set $\mathcal{S}$, tables~\ref{table:S:mu1}-\ref{table:S:mu1em3-mu1em6}, thereafter for problems in $\mathcal{M}$, tables~\ref{table:M:mu1}-\ref{table:M:mu1em6}, and finally for problems in $\mathcal{L}$, tables~\ref{table:L:mu1}-\ref{table:L:mu1em6}. The results are for three different regions depending on $\mu^0$, namely $\mu^0 = [1, \> 10^{-3}, \> 10^{-6}]$. The intention is to illustrate the performance of the modified Newton approach, both close to a solution and in a larger region where the theoretical results are not expected to hold. The results corresponding to $r = 16$ for problems in $\mathcal{S}$ are omitted due to similarity of the performance caused by the refactorization strategy. In all tables the initial factorization of $B^0 = F'(z^0)$ is counted as one factorization. In essence, ``1'' in the factorization column, $\texttt{F}$, means that no refactorization was performed. Moreover, ``-'' denotes that the method failed to converge within a maximum number of iterations. For each problem the maximum number of iterations was set to $10N$, where $N$ depends on the number of variables, see Appendix for the value of $N$ associated with each problem.
\begin{table}[H] 
\begin{center}
\begin{tiny}
\caption{Number of factorizations and iterations for problems in $\mathcal{S}$ with $\mu^0 = 1$.}\label{table:S:mu1}
\begin{tabular}{l|r|rrrrrr}
 & \multicolumn{1}{c}{\texttt{Newton}} & \multicolumn{2}{c}{\texttt{mN-r(2)}} & \multicolumn{2}{c}{\texttt{mN-r(2)-H1}} & \multicolumn{2}{c}{\texttt{mN-r(2)-H2}} \\
             & F/It & F    & It        & F         & It        & F         & It        \\ \hline
     \texttt{cvxqp1$\_$s}     &     13      &       2      &      57      &       1      &       43      &       1       &       41     \\
     \texttt{cvxqp2$\_$s}     &     14      &       2      &      54      &       2      &       54      &       2       &       54     \\
     \texttt{cvxqp3$\_$s}     &     17      &       2      &      74      &       2      &       59      &       2       &       58     \\
     \texttt{dual1}        &     21      &       5      &     188      &       4      &      132      &       3       &      128     \\
     \texttt{dual2}        &     19      &       4      &     150      &       3      &      109      &       3       &      110     \\
     \texttt{dual3}        &     20      &       4      &     168      &       3      &      127      &       3       &      126     \\
\end{tabular}
\end{tiny}
\end{center}
\end{table}
\begin{table}[H] 
\begin{tiny}
\begin{center}
\caption{Number of factorizations and iterations for problems in $\mathcal{S}$ with $\mu^0 = 10^{-3}$ to the left and $\mu^0 = 10^{-6}$ to the right.}\label{table:S:mu1em3-mu1em6}
\end{center}
\begin{tabular}{l|r|rrrrrr}
 & \multicolumn{1}{c}{\texttt{Newton}} & \multicolumn{2}{c}{\texttt{mN-r(2)}} & \multicolumn{2}{c}{\texttt{mN-r(2)-H1}} & \multicolumn{2}{c}{\texttt{mN-r(2)-H2}} \\
             & F/It & F    & It        & F         & It        & F         & It        \\ \hline
 \texttt{cvxqp1$\_$s}     &      5      &      1       &       5      &      1       &       5       &       1       &       5      \\
     \texttt{cvxqp2$\_$s}     &      5      &      1       &       5      &      1       &       5       &       1       &       5      \\
     \texttt{cvxqp3$\_$s}     &      6      &      1       &       7      &      1       &       7       &       1       &       7      \\
     \texttt{dual1}        &     12      &      3       &      86      &      2       &      51       &       2       &      51      \\
     \texttt{dual2}        &     10      &      2       &      50      &      2       &      50       &       2       &      49      \\
     \texttt{dual3}        &     11      &      2       &      58      &      2       &      57       &       2       &      57      \\
\end{tabular}
\quad
\begin{tabular}{l|r|rrrrrr}
 & \multicolumn{1}{c}{\texttt{Newton}} & \multicolumn{2}{c}{\texttt{mN-r(2)}} & \multicolumn{2}{c}{\texttt{mN-r(2)-H1}} & \multicolumn{2}{c}{\texttt{mN-r(2)-H2}} \\
             & F/It & F    & It        & F         & It        & F         & It        \\ \hline
     \texttt{cvxqp1$\_$s}     &     3       &      1       &       3      &      1       &       3       &       1       &       3      \\
     \texttt{cvxqp2$\_$s}     &     3       &      1       &       3      &      1       &       3       &       1       &       3      \\
     \texttt{cvxqp3$\_$s}     &     3       &      1       &       3      &      1       &       3       &       1       &       3      \\
     \texttt{dual1}        &     6       &      2       &      43      &      1       &      19       &       1       &      21      \\
     \texttt{dual2}        &     5       &      1       &       8      &      1       &       8       &       1       &       8      \\
     \texttt{dual3}        &     5       &      1       &      13      &      1       &      11       &       1       &      11      \\
\end{tabular}
\end{tiny}
\end{table}
\begin{table}[H]
\begin{tiny}
\begin{center}
\caption{Number of factorizations and iterations for problems in $\mathcal{M}$ with $\mu^0 = 1$.} \label{table:M:mu1}
\begin{tabular}{l|r|rrrrrr|rrrrrr}
 & \multicolumn{1}{c}{\texttt{Newton}} & \multicolumn{2}{c}{\texttt{mN-r(2)}} & \multicolumn{2}{c}{\texttt{mN-r(2)-H1}} & \multicolumn{2}{c}{\texttt{mN-r(2)-H2}} & \multicolumn{2}{c}{\texttt{mN-r(16)}} & \multicolumn{2}{c}{\texttt{mN-r(16)-H1}} & \multicolumn{2}{c}{\texttt{mN-r(16)-H2}}\\
                    & F/It& F  & It        & F         & It        & F         & It   & F         & It        & F         & It         & F         & It      \\ \hline
     \texttt{cvxqp1$\_$m}     &     15      &       2      &       140     &      2       &      112      &       2       &      111      &       3       &       29      &       3       &        29      &        3       &        29     \\
     \texttt{cvxqp2$\_$m}     &     13      &       2      &       103     &      2       &      102      &       2       &      102      &       2       &       24      &       2       &        24      &        2       &        24     \\
     \texttt{cvxqp3$\_$m}     &     16      &       3      &       204     &      2       &      127      &       2       &      121      &       3       &       37      &       3       &        33      &        3       &        34     \\
     \texttt{dualc8}       &     10      &       2      &        30     &      2       &       30      &       2       &       30      &       3       &       10      &       3       &        10      &        3       &        10     \\
     \texttt{gouldqp2}     &     15      &       4      &       213     &      3       &      168      &       3       &      194      &       4       &       32      &       4       &        31      &        4       &        30     \\
     \texttt{gouldqp3}     &     19      &       4      &       215     &      3       &      206      &       4       &      213      &       5       &       40      &       5       &        40      &        5       &        40     \\
\end{tabular}
\end{center}
\end{tiny}
\end{table}
\begin{table}[H]
\begin{tiny}
\begin{center}
\caption{Number of factorizations and iterations for problems in $\mathcal{M}$ with $\mu^0 = 10^{-3}$.} \label{table:M:mu1em3}
\begin{tabular}{l|r|rrrrrr|rrrrrr}
 & \multicolumn{1}{c}{\texttt{Newton}} & \multicolumn{2}{c}{\texttt{mN-r(2)}} & \multicolumn{2}{c}{\texttt{mN-r(2)-H1}} & \multicolumn{2}{c}{\texttt{mN-r(2)-H2}} & \multicolumn{2}{c}{\texttt{mN-r(16)}} & \multicolumn{2}{c}{\texttt{mN-r(16)-H1}} & \multicolumn{2}{c}{\texttt{mN-r(16)-H2}}\\
                    & F/It& F  & It        & F         & It        & F         & It   & F         & It        & F         & It         & F         & It      \\ \hline
     \texttt{cvxqp1$\_$m}     &      8      &      1       &       12     &      1       &       12      &       1       &       12      &       1       &        9      &       1       &         9      &        1       &        9      \\
     \texttt{cvxqp2$\_$m}     &      6      &      1       &        6     &      1       &        6      &       1       &        6      &       1       &        6      &       1       &         6      &        1       &        6      \\
     \texttt{cvxqp3$\_$m}     &      8      &      1       &       21     &      1       &       20      &       1       &       21      &       2       &       13      &       1       &        12      &        2       &       13      \\
     \texttt{dualc8}       &      6      &      1       &        6     &      1       &        6      &       1       &        6      &       2       &        6      &       2       &         6      &        2       &        6      \\
     \texttt{gouldqp2}     &     12      &      3       &      151     &      3       &      144      &       3       &      153      &       7       &       54      &       4       &        27      &        4       &       27      \\
     \texttt{gouldqp3}     &     14      &      3       &      145     &      2       &      130      &       2       &      131      &       4       &       28      &       4       &        28      &        4       &       28      \\
\end{tabular}
\end{center}
\end{tiny}
\end{table}
\vspace{-8mm}
\begin{table}[H]
\begin{tiny}
\begin{center}
\caption{Number of factorizations and iterations for problems in $\mathcal{M}$ with $\mu^0 = 10^{-6}$.} \label{table:M:mu1em6}
\begin{tabular}{l|r|rrrrrr|rrrrrr}
 & \multicolumn{1}{c}{\texttt{Newton}} & \multicolumn{2}{c}{\texttt{mN-r(2)}} & \multicolumn{2}{c}{\texttt{mN-r(2)-H1}} & \multicolumn{2}{c}{\texttt{mN-r(2)-H2}} & \multicolumn{2}{c}{\texttt{mN-r(16)}} & \multicolumn{2}{c}{\texttt{mN-r(16)-H1}} & \multicolumn{2}{c}{\texttt{mN-r(16)-H2}}\\
                    & F/It& F  & It        & F         & It        & F         & It   & F         & It        & F         & It         & F         & It      \\ \hline
     \texttt{cvxqp1$\_$m}     &      3      &      1       &       3      &      1       &       3       &       1       &       3       &       1       &       3       &       1       &        3       &        1       &        3      \\
     \texttt{cvxqp2$\_$m}     &      3      &      1       &       3      &      1       &       3       &       1       &       3       &       1       &       3       &       1       &        3       &        1       &        3      \\
     \texttt{cvxqp3$\_$m}     &      3      &      1       &       3      &      1       &       3       &       1       &       3       &       1       &       3       &       1       &        3       &        1       &        3      \\
     \texttt{dualc8}       &      3      &      1       &       3      &      1       &       3       &       1       &       3       &       1       &       3       &       1       &        3       &        1       &        3      \\
     \texttt{gouldqp2}     &      8      &      2       &      80      &      2       &      73       &       2       &      82       &       3       &      19       &       2       &       17       &        3       &       19      \\
     \texttt{gouldqp3}     &      5      &      1       &      15      &      1       &      14       &       1       &      16       &       2       &      10       &       2       &        9       &        2       &       10      \\
\end{tabular}
 \end{center}
\end{tiny}
\end{table}

\begin{table}[H]
\begin{tiny}
\begin{center}
\caption{Number of factorizations and iterations for problems in $\mathcal{L}$ with $\mu^0 = 1$.}  \label{table:L:mu1}
\begin{tabular}{l|r|rrrrrr|rrrrrr}
 & \multicolumn{1}{c}{\texttt{Newton}} & \multicolumn{2}{c}{\texttt{mN-r(2)}} & \multicolumn{2}{c}{\texttt{mN-r(2)-H1}} & \multicolumn{2}{c}{\texttt{mN-r(2)-H2}} & \multicolumn{2}{c}{\texttt{mN-r(16)}} & \multicolumn{2}{c}{\texttt{mN-r(16)-H1}} & \multicolumn{2}{c}{\texttt{mN-r(16)-H2}}\\
                    & F/It& F  & It        & F         & It        & F         & It   & F         & It        & F         & It         & F         & It      \\ \hline
 \texttt{aug2dcqp}     &     20      &       3      &      205     &       2      &      163      &       2       &      162      &       4       &       46      &       4       &        46      &        4       &        46     \\
     \texttt{aug2dqp}      &     25      &       3      &      305     &       3      &      244      &       3       &      244      &       6       &       69      &       6       &        69      &        6       &        68     \\
     \texttt{cont-050}     &     17      &       4      &       89     &       4      &       85      &       4       &       85      &       8       &       28      &       7       &        24      &        7       &        25     \\
     \texttt{cont-100}     &     19      &       4      &      322     &       4      &      312      &       4       &      316      &       7       &       85      &       7       &        81      &        7       &        81     \\
     \texttt{cont-101}     &     20      &       5      &      439     &       4      &      312      &       4       &      312      &       7       &       79      &       5       &        62      &        6       &        68     \\
     \texttt{cont-200}     &     20      &       4      &     1291     &       3      &      815      &       3       &      819      &       5       &      208      &       4       &       158      &        4       &       160     \\
\end{tabular}
\end{center}
\end{tiny}
\end{table}

\begin{table}[H]
\begin{tiny}
\begin{center}
\caption{Number of factorizations and iterations for problems in $\mathcal{L}$ with $\mu^0 = 10^{-3}$.} \label{table:L:mu1em3}
\begin{tabular}{l|r|rrrrrr|rrrrrr}
 & \multicolumn{1}{c}{\texttt{Newton}} & \multicolumn{2}{c}{\texttt{mN-r(2)}} & \multicolumn{2}{c}{\texttt{mN-r(2)-H1}} & \multicolumn{2}{c}{\texttt{mN-r(2)-H2}} & \multicolumn{2}{c}{\texttt{mN-r(16)}} & \multicolumn{2}{c}{\texttt{mN-r(16)-H1}} & \multicolumn{2}{c}{\texttt{mN-r(16)-H2}}\\
                    & F/It& F  & It        & F         & It        & F         & It   & F         & It        & F         & It         & F         & It      \\ \hline
 \texttt{aug2dcqp}     &     11      &      1       &       65     &      1       &       66      &       1       &       65      &       2       &       19      &       2       &        19      &        2       &        19     \\
     \texttt{aug2dqp}      &     16      &      2       &      142     &      2       &      142      &       2       &      142      &       4       &       39      &       4       &        39      &        4       &        39     \\
     \texttt{cont-050}     &     13      &      3       &       61     &      3       &       58      &       3       &       58      &       6       &       20      &       5       &        17      &        5       &        17     \\
     \texttt{cont-100}     &     15      &      3       &      219     &      3       &      209      &       3       &      213      &       5       &       59      &       5       &        55      &        4       &        42     \\
     \texttt{cont-101}     &     16      &      4       &      336     &      3       &      209      &       3       &      210      &       5       &       63      &       4       &        39      &        3       &        37     \\
     \texttt{cont-200}     &     17      &      4       &     1219     &      3       &      815      &       3       &      815      &       4       &      157      &       3       &       107      &        3       &       108     \\
\end{tabular}
\end{center}
\end{tiny}
\end{table}

\begin{table}[H]
\begin{tiny}
\begin{center}
\caption{Number of factorizations and iterations for problems in $\mathcal{L}$ with $\mu^0 = 10^{-6}$.} \label{table:L:mu1em6}
\begin{tabular}{l|r|rrrrrr|rrrrrr}
 & \multicolumn{1}{c}{\texttt{Newton}} & \multicolumn{2}{c}{\texttt{mN-r(2)}} & \multicolumn{2}{c}{\texttt{mN-r(2)-H1}} & \multicolumn{2}{c}{\texttt{mN-r(2)-H2}} & \multicolumn{2}{c}{\texttt{mN-r(16)}} & \multicolumn{2}{c}{\texttt{mN-r(16)-H1}} & \multicolumn{2}{c}{\texttt{mN-r(16)-H2}}\\
                    & F/It& F  & It        & F         & It        & F         & It   & F         & It        & F         & It         & F         & It      \\ \hline
     \texttt{aug2dcqp}     &      5      &      1       &       5      &      1       &        5      &       1       &        5      &       1       &       5       &       1       &        5       &        1       &        5      \\
     \texttt{aug2dqp}      &      7      &      1       &      35      &      1       &       35      &       1       &       35      &       2       &      15       &       2       &       15       &        2       &       15      \\
     \texttt{cont-050}     &      3      &      1       &       3      &      1       &        3      &       1       &        3      &       1       &       3       &       1       &        3       &        1       &        3      \\
     \texttt{cont-100}     &      6      &      1       &      87      &      1       &       87      &       1       &       87      &       2       &      14       &       2       &       14       &        2       &       14      \\
     \texttt{cont-101}     &      6      &      1       &      10      &      1       &       10      &       1       &        9      &       1       &      10       &       1       &       10       &        1       &        9      \\
     \texttt{cont-200}     &      8      &      2       &     407      &      2       &      406      &       2       &      406      &       2       &      53       &       2       &       52       &        2       &       52      \\
\end{tabular}
\end{center}
\end{tiny}
\end{table}
The results in tables~\ref{table:S:mu1}-\ref{table:L:mu1em6} indicate that the number of factorizations compared to those done by \texttt{Newton} may be reduced by instead performing low-rank updates, as with \texttt{mN-r($r$)}, $r = [2, \> 16]$. The reduced number of factorizations is however often at the expense of performing additional iterations with low-rank updates. The total number of iterations and/or factorizations are for many problems, but not for all, further reduced with heuristics \texttt{H1} and \texttt{H2}, as shown by the results corresponding to \texttt{mN-r($r$)-H1} and \texttt{mN-r($r$)-H2}, $r = [2, \> 16]$. This behavior is most significant in the simulations with larger values of $\mu$, as shown in Table~\ref{table:S:mu1}, Table~\ref{table:M:mu1} and Table~\ref{table:L:mu1}. 
Comparing the number of iterations in the \texttt{mN}-methods gives an indication of whether the heuristics have been active on a specific problem. The results in tables~\ref{table:S:mu1}-\ref{table:L:mu1em6} show that the performance of the heuristics varies with each problem, and in addition with $\mu^0$. In particular, the results show that the heuristics are most effective, and hence more likely to be active, at larger values of $\mu$. For smaller values of $\mu$, the \texttt{mN}-methods show similar performance, see Table~\ref{table:M:mu1em6}, Table~\ref{table:L:mu1em6} and the right-hand side of Table~\ref{table:S:mu1em3-mu1em6}. This indicates that the heuristics are less likely to have been active for smaller values of $\mu$. Consequently, the \texttt{mN}-methods are less likely to produce limiting steps at small values of $\mu$ on the benchmark problems. 
The observation is in line with the results of the theoretical sections.
Overall \texttt{mN-r(2)} fails to converge for two problems within a maximum number of iterations due to small step sizes. This is overcome with both \texttt{H1} and \texttt{H2}, as shown by the corresponding results in Table~\ref{table:M:mu1}. 

Numerical experiments have further shown that decreasing the refactorization parameters of Table~\ref{table:l} decreases the number of iterations, and increases the number of factorizations done by the \texttt{mN}-methods. In general, an increased rank of each update matrix reduces the number of iterations overall, but due to the refactorization strategy the methods are required to refactorize more often. 

Table~\ref{table:M:mu1em6}, Table~\ref{table:L:mu1em6} and the
right-hand side of Table~\ref{table:S:mu1em3-mu1em6} show that low-rank
updates give convergence for small values of $\mu$ in many of the
benchmark problems, even for update matrices of rank-two on large
scale problems.

Finally, we want to mention simplified Newton, i.e., \texttt{mN-r($0$)} without a refactorization strategy. Our simulations showed that this approach was significantly less robust. It is not clear to us how to deduce a refactorization strategy for \texttt{mN-r($0$)} that gives a fair comparison to the results of tables~\ref{table:S:mu1}-\ref{table:L:mu1em6}. Therefore, we have omitted specific results.

\section{Conclusion} \label{sec:conc}
In this work we have proposed and motivated a structured modified Newton approach
for solving systems of nonlinear equations that arise in interior-point methods for quadratic programming. In essence, the Jacobian of each Newton system is modified to a previous Jacobian plus one low-rank update matrix per succeeding iteration. The modified Jacobian maintains the sparsity pattern of the Jacobian and may thus be viewed as a Jacobian evaluated at a different point. The approach may in consequence be interpreted in the framework of previous works on primal-dual interior-point methods, e.g.,~effects of finite-precision arithmetic, stability, convergence and solution techniques. 

Numerical simulations have shown that small step sizes can have negative effects on convergence with the modified Newton approach, especially at larger values of $\mu$. In order to decrease these negative effects, we have constructed and motivated two heuristics. Further numerical simulations have shown that the two heuristics often increase the step sizes but also that this is not always sufficient to improve convergence. We have therefore also suggested a refactorization strategy. The heuristics and refactorization strategy that we have proposed are merely options, however, the framework allows for both different versions of these as well as other heuristics and/or strategies. 

In addition, we have performed numerical simulations on a set of convex quadratic benchmark problems. The results indicated that the number of factorizations compared to those of the Newton based method can be reduced, often at the expense of performing more iterations with low-rank updates. The total number of iterations and/or factorizations were for many problems, but not for all, further reduced with the two heuristics. Although the theoretical results are in the asymptotic region as $\mu \to 0$, or when the rank of each update matrix is sufficiently large, we still obtain interesting numerical results for larger values of $\mu$ and update matrices of low-rank. 

Our work is meant to contribute to the theoretical and numerical
understanding of modified Newton approaches for solving systems of
nonlinear equations that arise in interior-point methods. We have laid
a foundation that may be adapted and included in more sophisticated
interior-point solvers as well as contribute to the development of
preconditioners. We have limited ourselves to a numerical study of the
accuracy of the approaches in a high-level language. To get a full
understanding of the practical performance, precise ways of solving
the updated modified Newton systems would have to be investigated
further.

\appendix 
\section{Appendix}
For completeness we state the Eckart-Young-Mirsky theorem.
\begin{theorem}[Eckart-Young-Mirsky theorem] \label{thm:EckartYoundMirsky}
Let $A \in \mathbb{R}^{m \times n}$  be of rank $r$ and denote its singular value decomposition by $U \Sigma V^T$, where $U \in \mathbb{R}^{m \times m}$, $V \in \mathbb{R}^{n \times n}$ are orthogonal matrices, and $\Sigma = \diag(\sigma) \in \mathbb{R}^{m \times n}$, for $\sigma_1  \geq \dots \geq \sigma_p \geq 0 $, where $p = \min\{m,n\}$. For a given $q$, $0 < q \le r$, the optimal solution of
\begin{equation*}
\begin{array}{cl} 
\minimize{\tilde A \in \mathbb{R}^{m \times n}} & {\| A - \tilde A \|} \\
 \subject &  \rank{ ( \tilde A )} \le q,
\end{array} 
\end{equation*}
where $\| . \|$ is either 2-norm or Frobenius norm, is \[
A_* = \sum_{i=1}^q \sigma_i u_i v_i^T, \]
with $u_i$ and $v_i$ as the $i$th column of $U$ and $i$th column of $V$ respectively. 
\end{theorem}
\begin{proof}
See \cite[Ch.~2]{GolVan13}.
\end{proof}

The following lemma contains the singular value decomposition of $\Delta F'( \Delta z)$ given in (\ref{eq:dFp}). 
\begin{lemma} \label{lemma:SVDdFp}
For $\Delta z = ( \Delta x, \Delta \lambda, \Delta s) \in \mathbb{R}^{(n+2m)}$ let $\Delta F'( \Delta z) \in \mathbb{R}^{(n+2m) \times (n+2m)}$  be defined by (\ref{eq:dFp}). 
The singular value decomposition of $\Delta F'( \Delta z)$ can then be written as \[
\Delta F'( \Delta z)  = \sum_{i\in \mathcal{V}} e_{n+m+i} \left( \Delta s_i e_{n+i} + \Delta \lambda_i e_{m+n+i} \right)^T,
\]
where $\mathcal{V}$ is a set of indices, $i=1,\dots, m$, ordered such that $\sqrt{(\Delta \lambda_i)^2 + (\Delta s_i)^2}$ are in descending order.
\end{lemma}
\begin{proof}
The left singular vectors of $\Delta F'( \Delta z)$ are the set of orthonormal eigenvectors of $(\Delta F'(\Delta z)) (\Delta F'(\Delta z))^T$, i.e., vectors $u$ such that 
\begin{equation} \label{eq:lemma:SVD:proof:eq1}
 \begin{pmatrix} 0 & 0 & 0 \\ 0 & 0 & 0\\ 0 & 0 &  (\Delta S)^2 + (\Delta \Lambdait)^2 \end{pmatrix} u = \tilde{\lambda} u.
  \end{equation}
The eigenvectors of  $(\Delta F'(\Delta z)) (\Delta F'(\Delta z))^T$ are $e_i$, $i=1,\dots, n + 2m$ and the eigenpairs, with nonzero eigenvalues, are $( (\Delta \lambda _i)^2 + (\Delta s_i)^2, e_{n+m+i})$, $i =1, \dots, m$. Similarly, the right singular vectors are the set of orthonormal eigenvectors of $(\Delta F'(\Delta z))^T (\Delta F'(\Delta z))$, i.e., vectors $v$ such that 
\begin{equation} \label{eq:lemma:SVD:proof:eq2}
 \begin{pmatrix} 0 & 0 & 0 \\ 0 & (\Delta S)^2 & \Delta S \Delta \Lambdait \\ 0 & \Delta \Lambdait \Delta S &  (\Delta \Lambdait)^2 \end{pmatrix} v = \tilde{\lambda} v.
 \end{equation}
The nonzero eigenvalues of (\ref{eq:lemma:SVD:proof:eq2}) are the same as those in (\ref{eq:lemma:SVD:proof:eq1}). A straightforward calculation shows that the eigenvector corresponding to the $i$th nonzero eigenvalue $v_i = \frac{1}{\sqrt{(\Delta \lambda_i)^2 + (\Delta s_i)^2}}\left( \Delta s_i e_{n+i} + \Delta \lambda_i e_{m+n+i} \right)$, $i=1,\dots,m$, fulfills (\ref{eq:lemma:SVD:proof:eq2}) with ${\tilde{\lambda}_i =(\Delta \lambda _i)^2 +( \Delta s_i)^2}$, $i=1,\dots,m$, and in addition that the set of vectors $v_i$, $i=1,\dots,m$, form an orthornormal set. The singular values of $\Delta F'( \Delta z)$ are then given by $\sqrt{\tilde \lambda_i}$,  $i = 1, \dots,  n + 2m$.
\end{proof}
 
 \subsection*{Proof of Theorem~\ref{thm:dzErr}}
 \begin{proof} $F'(z^+)$ is nonsingular since $z^+ \in \mathcal{B}(z^*, \delta)$. The modified Jacobian can then be written as \[B^+ = F'(z^+)(I-F'(z^+)\inv E^+).\]  By Lemma~\ref{lemma:spectralradiusMot}, $\| F'(z^+) \inv E^+ \|   \le M  C^{(r+1)} \mu$ and hence there exists $\mubar>0$, with $\mubar \le \muhat$, such that for $0<\mu \le \mubar$ it holds that $M  C^{(r+1)} \mu < 1$. In consequence, for $0<\mu \le \mubar$,  $(B^+)^{-1}$ can be expanded as a von Neumann series
\begin{align} \label{eq:thm:proof:1}
(B^+) \inv  & = \left( I - F'(z^+)\inv E^+ \right) \inv F'(z^+) \inv \nonumber \\ 
& = \sum_{j=0}^\infty \left( F'(z^+) \inv E^+ \right)^j  F'(z^+) \inv.
\end{align}
The error with respect to the Newton step $\Delta \hat{z}^+$ can with (\ref{eq:thm:proof:1}) be written as
\begin{align} \label{eq:thm:proof:2}
\Delta \hat{z}^+ - \Delta z^+ & = \Delta \hat{z}^+ + (B^+) \inv F_{\mu^{++}}(z^+)  = \big( I - \big( I - F'(z^+)\inv E^+ \big) \inv \big) \Delta \hat{z}^+\nonumber \\
&= \big( I -  \sum_{j=0}^\infty \big( F'(z^+) \inv E^+ \big)^j \big) \Delta \hat{z}^+ = - \sum_{j=1}^\infty \big( F'(z^+) \inv E^+ \big)^j \Delta \hat{z}^+.
\end{align}
Taking 2-norm on both sides of (\ref{eq:thm:proof:2}) and making use
of norm inequalities give
\begin{equation} \label{eq:thm:proof:3}
\| \Delta \hat{z}^+ - \Delta z^+ \| \le  \sum_{j =1}^\infty \| F'(z^+) \inv E^+ \|^j  \| \Delta \hat{z}^+ \|.
\end{equation}
The sum in (\ref{eq:thm:proof:3}) is a geometric series which is convergent since $\| F'(z^+)\inv E^+ \| < 1$ for $\mu \le \mubar$ and hence
\begin{equation} \label{eq:thm:proof:4}
\| \Delta \hat{z}^+ - \Delta z^+  \| \le  \frac{\| F'(z^+) \inv E^+ \|}{1-\| F'(z^+) \inv E^+ \|}  \| \Delta \hat{z}^+ \|.
\end{equation}
Recall that $\Delta \hat{z}^+$ is the solution of $ F'(z^+) \Delta \hat{z}^+ = - F_{\mu^{++}}(z^+)$ where $z^+ \in \mathcal{B}(z^*, \delta)$ and hence
\begin{align}  \label{eq:thm:proof:6} \
\| \Delta \hat{z}^+ \| & = \|   -F'(z^+) \inv  F_{\mu^{++}}(z^+) \|  \nonumber \\
\nonumber 
& =  \|  F'(z^+) \inv \big(  F_{\mu^{++}}(z^{\mu^{++}}) - F_{\mu^{++}}(z^+) - F'(z^+) (z^{\mu^{++}} - z^+)  \big)  +  (z^{\mu^{++}} - z^+) \|  \nonumber \\
& \leq \frac{M}{2}  \| z^+ - z^{\mu^{++}} \|^2 + \| z^+ - z^{\mu^{++}} \|,
\end{align}
where $z^{\mu^{++}} : F_{\mu^{++}} (z^{\mu^{++}}) = 0$, and it has been used that the Lipschitz constant of $F'$ equals one, as discussed in the end of Section~\ref{sec:genk}. Moreover,
\begin{align} \label{eq:A7}
\| z^+ - z^{\mu^{++}} \| & = \| z - F'(z)\inv F_{\mu^+}  (z)  - z^{\mu^{++}}  \| \nonumber \\
& =  \| F'(z) \inv \big(  F_{\mu^+} (z^{\mu^+}) - F_{\mu^+} (z) - F'(z) ( z^{\mu^+} - z)  \big) +  ( z^{\mu^+} - z^{\mu^{++}}) \| \nonumber \\
& \leq \frac{M}{2} \|z^{\mu^+} - z\|^2 + \|  z^{\mu^+} - z^{\mu^{++}} \| \nonumber \\
& =  \frac{M}{2} \|z^{\mu^+} -z^\mu + z^\mu - z\|^2 + \|  z^{\mu^+} - z^{\mu^{++}} \| \nonumber \\
& \leq   \frac{M}{2}  \big(  \| z - z^{\mu} \|^2 + 2  \| z -z^{\mu} \|\| z^\mu -z^{\mu^+} \| + \|z^{\mu} - z^{\mu^+} \|^2  \big) \nonumber \\
& \quad +\|  z^{\mu^+} - z^{\mu^{++}} \|, 
\end{align}
where  $z^{\mu^{+}}: F_{\mu^{+}} (z^{\mu^{+}}) = 0$.
Note that $\| z^\mu - z^{\mu^+ } \| \leq C_3 (\mu - \mu^+) = C_3 (1-\sigma) \mu$, as  $z^{\mu}$,  $0 < \mu \leq \muhat$,  is Lipschitz continuous by Lemma~\ref{lemma:LipcPath}.
By assumption there exist a constant $C>0$ such that $\| F_\mu (z) \| \le C \mu$. Lemma~\ref{lemma:FmuBound} and Lemma~\ref{lemma:LipcPath} applied to (\ref{eq:A7}) then give
\begin{align} \label{eq:thm:proof:7}
\| z^+ - z^{\mu^{++}} \| & \leq \frac{M}{2}  \left( \frac{C^2}{C_4^2}  + 2  \frac{C}{C_4} C_3 (1-\sigma)  +  C_3^2 (1-\sigma)^2   \right) \mu^2 +   C_3(1 - \sigma^+) \sigma \mu \nonumber \\
& = C_3(1 - \sigma^+) \sigma \mu + \mathcal{O}(\mu^2).
\end{align}
A combination of (\ref{eq:thm:proof:6}) and (\ref{eq:thm:proof:7}) gives
\begin{equation} \label{eq:thm:proof:8n}
\| \Delta \hat{z}^+ \| \leq C_3(1 - \sigma^+) \sigma \mu + \mathcal{O}(\mu^2).
\end{equation}
Insertion of (\ref{eq:thm:proof:8n}) into (\ref{eq:thm:proof:4}) while making use of Lemma~\ref{lemma:spectralradiusMot} gives the result.
 \end{proof}

\subsection*{Pre-processing of benchmark problems}
As the benchmark problems in general also contain equality constraints, each problem was pre-processed and put on the form 
\begin{equation} \label{eq:optPbProb}
\begin{array}{cl} 
\minimize{x\in \mathbb{R}^n} & { \frac{1}{2} x^T H x + c^T x} \\
 \subject & A_{eq} x = b_{eq}  \\
                  & A_{in}  x \ge b_{in} ,
\end{array} 
\end{equation}
where $H \in \mathbb{R}^{n \times n}$, $c \in \mathbb{R}^n$, $A_{eq} \in \mathbb{R}^{m_{eq} \times n }$, $b_{eq} \in \mathbb{R}^{m_{eq}}$, $A_{in} \in \mathbb{R}^{m_{in} \times n }$, and $b_{in} \in \mathbb{R}^{m_{in}}$. First-order necessary conditions for a local minimizer of (\ref{eq:optPbProb}) can be stated as: 
i)~$Hx + c -  A_{eq}^T \lambda_{eq} - A_{in}^T \lambda_{in} = 0$, 
ii)~$A_{eq} x = b_{eq}$, 
iii)~$A_{in} x - s = b_{in} $. 
iv)~$s \cdot \lambda_{in} = 0$,
v)~$s \geq  0$, 
vi)~$\lambda_{in} \ge 0$, for vectors 
$\lambda_{eq} \in \mathcal{R}^{m_{eq}}$,  $\lambda_{in} \in \mathcal{R}^{m_{in}}$ and $s \in \mathcal{R}^{m_{in}}$. Similarly as in Section~\ref{sec:Background}, define  $F_{\mu}:\mathbb{R}^{n+m_{eq}+2m_{in}} \rightarrow \mathbb{R}^{n+m_{eq}+2m_{in}} $ by
\begin{equation} \label{eq:benchmarkproblemsFmu}
F_{\mu}(z) = \begin{pmatrix}
Hx +c - A_{eq}^T \lambda_{eq}-  A_{in}^T\lambda_{in} \\
A_{eq} x - b_{eq} \\
A_{in} x-s-b_{in} \\
\Lambdait_{in} S e - \mu e
\end{pmatrix}, \mbox{ with }z = (x, \lambda_{eq}, \lambda_{in},s).
\end{equation}
Primal-dual interior-point methods involve solving or approximately solving  $F_{\mu}(z) = 0$ for a decreasing sequence of $\mu>0$ while maintaining $\lambda_{in} > 0$ and $s>0$. Application of Newton iterations gives systems on the form (\ref{eq:PDnewtSyst}) with $F_\mu(z)$ as in (\ref{eq:benchmarkproblemsFmu}), 
$\Delta \hat{z} = (\Delta \hat{x}, \Delta \hat{\lambda}_{eq},\Delta \hat{\lambda}_{in},  \Delta \hat{s})$ and $F': \mathbb{R}^{n+m_{eq}+2m_{in}} \rightarrow \mathbb{R}^{(n+m_{eq}+2m_{in}) \times (n+m_{eq}+2m_{in})}$ defined by
\begin{equation}\label{App:eq:Fp}
F'(z)  = \begin{pmatrix}
H & -A_{eq}^T &  -A_{in}^T & \\
A_{eq} &   &   &  \\
A_{in} &   &   & -I \\
  &   &  S & \Lambdait_{in}
\end{pmatrix}.
\end{equation}

\subsection*{Problem data}
Number of $x$-variables, equality constraints, inequality constraints and total number of variables in the primal-dual formulation for problems in the sets $\mathcal{S}$, $\mathcal{M}$ and $\mathcal{L}$ are shown in Table~\ref{table:probData:S} and Table~\ref{table:probData:ML} respectively.
\begin{table}[H]
\begin{tiny}
\begin{center}
\caption{Details on problem size for problems $p \in \mathcal{S}$.} \label{table:probData:S}
\begin{tabular}{l|rrrr} 
  & \multicolumn{1}{c}{$n$} &   \multicolumn{1}{c}{$m_{eq}$}  &  \multicolumn{1}{c}{$m_{in}$}   &  \multicolumn{1}{c}{$N$}  \\ \hline
     \texttt{cvxqp1$\_$s}     &    100      &      50      &     200      &     350     \\
     \texttt{cvxqp2$\_$s}     &    100      &      25      &     200      &     325     \\
     \texttt{cvxqp3$\_$s}     &    100      &      75      &     200      &     375     \\
     \texttt{dual1}        &     85      &       1      &     170      &     256     \\
     \texttt{dual2}        &     96      &       1      &     192      &     289     \\
     \texttt{dual3}        &    111      &       1      &     222      &     334     \\
     \texttt{dual4}        &     75      &       1      &     150      &     226     \\
     \texttt{dualc1}       &      9      &       1      &     232      &     242     \\
     \texttt{dualc2}       &      7      &       1      &     242      &     250     \\
     \texttt{dualc5}       &      8      &       1      &     293      &     302     \\
     \texttt{qafiro}       &     32      &       8      &      51      &      91     \\
     \texttt{hs118}        &     15      &       0      &      59      &      74     \\
     \texttt{hs268}        &      5      &       0      &       5      &      10     \\
     \texttt{hs53}         &      5      &       3      &      10      &      18     \\
     \texttt{hs76}         &      4      &       0      &       7      &      11     \\
     \texttt{lotschd}      &     12      &       7      &      12      &      31     \\
     \texttt{primal1}      &    325      &       0      &      86      &     411     \\
     \texttt{primalc1}     &    230      &       0      &     224      &     454     \\
     \texttt{primalc2}     &    231      &       0      &     236      &     467     \\
     \texttt{qadlittl}     &     96      &      14      &     137      &     247     \\
     \texttt{qisrael}      &    142      &       0      &     316      &     458     \\
     \texttt{qpcblend}     &     83      &      43      &     114      &     240     \\
     \texttt{qscagr7}      &    140      &      84      &     185      &     409     \\
     \texttt{qshare2b}     &     79      &      13      &     162      &     254     \\
     \texttt{s268}         &      5      &       0      &       5      &      10      
\end{tabular}
\end{center}
\end{tiny}
\end{table}

\begin{table}[H]
\begin{tiny}
\begin{center}
\caption{Details on problem size for problems $p \in \mathcal{M}$ and $p \in \mathcal{L}$.} \label{table:probData:ML}
\begin{tabular}{l|rrrr} 
  & \multicolumn{1}{c}{$n$} &   \multicolumn{1}{c}{$m_{eq}$}  &  \multicolumn{1}{c}{$m_{in}$}   &  \multicolumn{1}{c}{$N$}  \\ \hline
     \texttt{cvxqp1$\_$m}     &    1000     &      500     &     2000     &     3500    \\
     \texttt{cvxqp2$\_$m}     &    1000     &      250     &     2000     &     3250    \\
     \texttt{cvxqp3$\_$m}     &    1000     &      750     &     2000     &     3750    \\
     \texttt{dualc8}       &       8     &        1     &      518     &      527    \\
     \texttt{gouldqp2}     &     699     &      349     &     1398     &     2446    \\
     \texttt{gouldqp3}     &     699     &      349     &     1398     &     2446    \\
     \texttt{ksip}         &      20     &        0     &     1001     &     1021    \\
     \texttt{laser}        &    1002     &        0     &     2000     &     3002    \\
     \texttt{primal2}      &     649     &        0     &       97     &      746    \\
     \texttt{primal3}      &     745     &        0     &      112     &      857    \\
     \texttt{primal4}      &    1489     &        0     &       76     &     1565    \\
     \texttt{primalc5}     &     287     &        0     &      286     &      573    \\
     \texttt{primalc8}     &     520     &        0     &      511     &     1031    \\
     \texttt{q25fv47}      &    1571     &      515     &     1876     &     3962    \\
     \texttt{qgrow15}      &     645     &      300     &     1245     &     2190    \\
     \texttt{qgrow22}      &     946     &      440     &     1826     &     3212    \\
     \texttt{qgrow7}       &     301     &      140     &      581     &     1022    \\
     \texttt{qshell}       &    1525     &      534     &     1644     &     3703    \\
     \texttt{qpcstair}     &     385     &      209     &      532     &     1126    \\
     \texttt{qcapri}       &     337     &      142     &      583     &     1062    \\
     \texttt{qsctap1}      &     480     &      120     &      660     &     1260    \\
     \texttt{qsctap2}      &    1880     &      470     &     2500     &     4850    \\
     \texttt{qsctap3}      &    2480     &      620     &     3340     &     6440    \\
     \texttt{qsc205}       &     202     &       90     &      315     &      607    \\
     \texttt{qscagr25}     &     500     &      300     &      671     &     1471    \\
     \texttt{qscsd1}       &     760     &       77     &      760     &     1597    \\
     \texttt{qscsd6}       &    1350     &      147     &     1350     &     2847    \\
     \texttt{qscsd8}       &    2750     &      397     &     2750     &     5897    \\
     \texttt{qshare1b}     &     225     &       89     &      253     &      567    \\
     \texttt{values}       &     202     &        1     &      404     &      607    \\
     \texttt{aug3dcqp}     &    3873     &     1000     &     3873     &     8746    \\
     \texttt{aug3dqp}      &    3873     &     1000     &     3873     &     8746    \\
     \texttt{stadat1}      &    2001     &        0     &     5999     &     8000    \\
     \texttt{stadat2}      &    2001     &        0     &     5999     &     8000    \\
     \texttt{mosarqp1}     &    2500     &        0     &     3200     &     5700    \\
     \texttt{mosarqp2}     &     900     &        0     &     1500     &     2400    \\
     \texttt{yao}     &    2000     &        0     &     2001     &     4001    
\end{tabular}
\quad
\begin{tabular}{l|rrrr}
  & \multicolumn{1}{c}{$n$} &   \multicolumn{1}{c}{$m_{eq}$}  &  \multicolumn{1}{c}{$m_{in}$}   &  \multicolumn{1}{c}{$N$}  \\ \hline
     \texttt{aug2dcqp}     &    20200     &     10000     &     20200     &      50400    \\
     \texttt{aug2dqp}      &    20200     &     10000     &     20200     &      50400    \\
     \texttt{cont-050}     &     2597     &      2401     &      5194     &      10192    \\
     \texttt{cont-100}     &    10197     &      9801     &     20394     &      40392    \\
     \texttt{cont-101}     &    10197     &     10098     &     20394     &      40689    \\
     \texttt{cont-200}     &    40397     &     39601     &     80794     &     160792    \\
     \texttt{cont-201}     &    40397     &     40198     &     80794     &     161389    \\
     \texttt{stcqp2}       &     4097     &      2052     &      8194     &      14343    \\
     \texttt{stadat3}      &     4001     &         0     &     11999     &      16000    \\
     \texttt{cvxqp1$\_$l}     &    10000     &      5000     &     20000     &      35000    \\
     \texttt{cvxqp2$\_$l}     &    10000     &      2500     &     20000     &      32500    \\
     \texttt{cvxqp3$\_$l}     &    10000     &      7500     &     20000     &      37500    \\
     \texttt{exdata}       &     3000     &         1     &      7500     &      10501    \\
     \texttt{hues-mod}     &    10000     &         2     &     10000     &      20002    \\
     \texttt{huestis}      &    10000     &         2     &     10000     &      20002    \\
     \texttt{liswet1}      &    10002     &         0     &     10000     &      20002    \\
     \texttt{liswet2}      &    10002     &         0     &     10000     &      20002    \\
     \texttt{liswet3}      &    10002     &         0     &     10000     &      20002    \\
     \texttt{liswet4}      &    10002     &         0     &     10000     &      20002    \\
     \texttt{liswet5}      &    10002     &         0     &     10000     &      20002    \\
     \texttt{liswet6}      &    10002     &         0     &     10000     &      20002    \\
     \texttt{liswet7}      &    10002     &         0     &     10000     &      20002    \\
     \texttt{liswet8}      &    10002     &         0     &     10000     &      20002    \\
     \texttt{liswet9}      &    10002     &         0     &     10000     &      20002    \\
     \texttt{liswet10}     &    10002     &         0     &     10000     &      20002    \\
     \texttt{liswet11}     &    10002     &         0     &     10000     &      20002    \\
     \texttt{liswet12}     &    10002     &         0     &     10000     &      20002    \\
     \texttt{ubh1}         &    17997     &     12000     &     12006     &      42003    
\end{tabular}
\end{center}
\end{tiny}
\end{table}

\bibliography{references,references2}
\bibliographystyle{myplain}

\end{document}